\documentclass[reqno]{amsart}
\usepackage[a4paper, margin=2cm] {geometry}
\usepackage{etoolbox}

\patchcmd{\section}{\scshape\centering}{\bfseries\Large}{}{}
\makeatletter
\renewcommand{\@secnumfont}{\bfseries\Large}
\makeatother

\makeatletter

\def\enddoc@text{\ifx\@empty\@translators \else\@settranslators\fi}

\g@addto@macro\@setauthors{
  \par\vspace{2ex} 
  \@setaddresses
}
\makeatother

\patchcmd{\subsection}{\bfseries}{\bfseries\large}{}{}
\patchcmd{\subsection}{-.5em}{.5ex plus .2ex}{}{}

\usepackage[utf8]{inputenc}
\usepackage{amsmath}
\usepackage{siunitx}
\usepackage{amsfonts}
\usepackage{graphicx}
\usepackage{wrapfig}
\usepackage{subfig}
\usepackage{listings}
\usepackage{mathabx}
\usepackage{hyperref}
\usepackage{color} 
\usepackage{amssymb,amscd}
\usepackage{enumerate}
\usepackage{etoolbox}
\usepackage{tikz}
\usepackage{faktor}
\usepackage{graphicx}
\usepackage[utf8]{inputenc}
\usepackage{amsmath}
\usepackage{siunitx}
\usepackage{amsfonts}
\usepackage{graphicx}
\usepackage{wrapfig}
\usepackage{subfig}
\usepackage{listings}
\usepackage{mathabx}
\usepackage{scalerel}
\usepackage{stackengine}
\usepackage{color}
\usepackage{thmtools}
\usepackage{thm-restate}
\usepackage{mathtools}
\usepackage{enumitem}

\usepackage{overpic} 
\usepackage{url} 
\usepackage{comment} 
\newcommand\reallywidetilde[1]{\ThisStyle{
\setbox0=\hbox{$\SavedStyle#1$}
\stackengine{-.1\LMpt}{$\SavedStyle#1$}{
\stretchto{\scaleto{\SavedStyle\mkern.2mu\sim}{.5150\wd0}}{.6\ht0}
}{O}{c}{F}{T}{S}
}}

\newcommand{\rr}{\mathbb{R}}

\newcommand{\qq}{\overline{q}}

\numberwithin{equation}{section}

\newtheorem{theorem}{Theorem}
\newtheorem{lemma}[theorem]{Lemma}
\newtheorem{corollary}[theorem]{Corollary}
\newtheorem{proposition}[theorem]{Proposition}
\newtheorem{observation}[theorem]{Observation}
\theoremstyle{definition}

\newtheorem{problem}{Problem}

\newtheorem{definition}[theorem]{Definition}

\newcommand{\nn}{\overline{n}}
\newcommand{\mm}{\overline{m}}

\newcommand{\pp}{\overline{p}}

\newcommand{\lip}{\mathrm{Lip}_0}

\newcommand{\comp}{\overset{c}{\hookrightarrow}}
\newcommand{\FF}{\mathcal{F}}
\newcommand{\Ext}{\mathrm{Ext}}
\newcommand{\dk}{d^{\lVert \cdot \rVert_\infty}_k}
\newcommand{\Lip}{\mathrm{Lip}}

\newcommand{\dindex}[4][-1pt]{{#2}_{#3}^{\hspace{#1}#4}}
\newcommand{\bigslant}[2]{
  \mathchoice
    {\left.\raisebox{.2em}{$\displaystyle #1$}\middle/\raisebox{-.2em}{$\displaystyle #2$}\right.}
    {\left.\raisebox{.2em}{$\textstyle #1$}\middle/\raisebox{-.2em}{$\textstyle #2$}\right.}
    {\left.\raisebox{.1em}{$\scriptstyle #1$}\middle/\raisebox{-.1em}{$\scriptstyle #2$}\right.}
    {\left.\raisebox{.05em}{$\scriptscriptstyle #1$}\middle/\raisebox{-.05em}{$\scriptscriptstyle #2$}\right.}
}
\date{}
\usepackage{amsrefs}

\hypersetup{colorlinks=false}

\usepackage[nameinlink]{cleveref}

\crefformat{equation}{\begingroup\hypersetup{linkbordercolor=cyan}#2eq.~(#1)#3\endgroup}
\Crefformat{equation}{\begingroup\hypersetup{linkbordercolor=cyan}#2Eq.~(#1)#3\endgroup}

\crefformat{theorem}{\begingroup\hypersetup{linkbordercolor=red}#2theorem~#1#3\endgroup}
\Crefformat{theorem}{\begingroup\hypersetup{linkbordercolor=red}#2Theorem~#1#3\endgroup}

\crefformat{lemma}{\begingroup\hypersetup{linkbordercolor=red}#2lemma~#1#3\endgroup}
\Crefformat{lemma}{\begingroup\hypersetup{linkbordercolor=red}#2Lemma~#1#3\endgroup}

\crefformat{proposition}{\begingroup\hypersetup{linkbordercolor=red}#2proposition~#1#3\endgroup}
\Crefformat{proposition}{\begingroup\hypersetup{linkbordercolor=red}#2Proposition~#1#3\endgroup}

\crefformat{observation}{\begingroup\hypersetup{linkbordercolor=red}#2observation~#1#3\endgroup}
\Crefformat{observation}{\begingroup\hypersetup{linkbordercolor=red}#2Observation~#1#3\endgroup}

\newcommand{\eqcref}[1]{%
  \begingroup
  \hypersetup{linkbordercolor=cyan}%
  \eqref{#1}%
  \endgroup
}

\makeatletter

\pretocmd{\@setthanks}{\par\noindent}{}{}
\makeatother
\title[Lipschitz-free spaces over products of sequences]{Lipschitz-free spaces over products of sequences}
\author{Fraser Mason}
\address{Department of Pure Mathematics and Mathematical Statistics, Centre for Mathematical Sciences, University of Cambridge, 
  Wilberforce Road, Cambridge CB3 0WB, United Kingdom}
\email{fm531@cam.ac.uk}

\begin{document}

\begin{abstract}
    \noindent We answer positively a question of Aliaga and show that for any nonconstant real polynomial $p$, the Lipschitz-free space over $\{(p(n), p(m)):n, m\in \mathbb{N}\}$ is isomorphic to $\FF(\mathbb{Z}^2)$. We in fact show more generally that if $d\in \mathbb{N}$, $q\in \mathbb{Z}_{\geq 0}$, and $((\dindex{a}{n}{(i)})_{n=1}^\infty)_{i=1}^d$, $((\dindex{b}{m}{(j)})_{m=1}^\infty)_{j=1}^q$ are sequences with $0<\dindex{a}{1}{(i)}<\dindex{a}{2}{(i)}<\cdots$, $0<\dindex{b}{1}{(j)}<\dindex{b}{2}{(j)}<\cdots$, $\dindex{a}{n}{(i)}\to \infty$ as $n\to \infty$, $\frac{\dindex{a}{n+1}{(i)}}{\dindex{a}{n}{(i)}}\to 1$ as $n\to \infty$ and $\underset{m\to \infty}{\liminf}{\frac{\dindex{b}{m+1}{(j)}}{\dindex{b}{m}{(j)}}}>1$ for all $i, j$, then the Lipschitz-free space over the product of these $d+q$ sequences is isomorphic to $\FF(\mathbb{Z}^d)$.
\end{abstract}
\maketitle

\section{Introduction} 

\subsection{Background, motivation and main result}

\noindent Given a metric space $M$ and a basepoint $0_M\in M$, $\lip(M)$ denotes the space of Lipschitz functions $f:M\to \mathbb{R}$ with $f(0_M)=0$, and is a Banach space in the Lipschitz norm. For $x\in M$, the map $\delta_M(x):\lip(M)\to \mathbb{R}$ defined by $\delta_M(x)(f):=f(x)$ is an element of $\lip(M)^{\star}$, and the \emph{Lipschitz-free space} $\FF(M)$ is defined to be the subspace $\overline{\mathrm{span}}^{\lVert \cdot \rVert}\{\delta_M(x):x\in M\}\subset \lip(M)^{\star}$. The isometry class of both spaces is independent of basepoint. 
\noindent In this paper, we provide a positive solution to a problem raised by Aliaga on whether for every nonconstant real polynomial $p$, the Lipschitz-free space over $\{(p(m), p(n)):m, n\in \mathbb{N}\}$  is isomorphic to that over $\mathbb{Z}^2$ (see \cite{Aliaga}*{Question 6}). This follows from a particular case of the following theorem with $q=0$ and $d=2$, after first replacing $p$ by $-p$ if necessary to ensure a positive leading coefficient, adding a sufficiently large constant to $p$ to ensure that $p(n)>0$ for all $n$, and taking the strictly increasing enumeration of $(p(n))_{n=1}^\infty$.\\

\begin{restatable}{theorem}{thmone}
\label{thm1} Let $d,q\in \mathbb{Z}_{\geq 0}$ with $d+q\geq 1$. For each $1\leq i\leq d$, let $(\dindex{a}{n}{(i)})_{n=1}^\infty$ be a sequence such that $0<\dindex{a}{1}{(i)}<\dindex{a}{2}{(i)}<\cdots$, $\dindex{a}{n}{(i)}\to \infty$ as $n\to \infty$, and $\frac{\dindex{a}{n+1}{(i)}}{\dindex{a}{n}{(i)}}\to 1$ as $n\to \infty$. For each $1\leq j\leq q$, let $(\dindex{b}{m}{(j)})_{m=1}^\infty$ be a sequence such that $0<\dindex{b}{1}{(j)}<\dindex{b}{2}{(j)}<\cdots$ and $\underset{m\to \infty}{\liminf}{\frac{\dindex{b}{m+1}{(j)}}{\dindex{b}{m}{(j)}}}>1$. Let \[N=\{(\dindex{a}{n_1}{(1)}, \cdots, \dindex{a}{n_d}{(d)}, \dindex{b}{m_1}{(1)}, \cdots, \dindex{b}{m_q}{(q)}):n_i, m_j\in \mathbb{N}\}\subset \mathbb{R}^{d+q}.\]Then \[\FF(N)\sim \begin{cases}\FF(\mathbb{Z}^d) & \text{ if } d>0 \\ \ell_1 &\text{ if } d=0 \text{ and }q>0. \end{cases}\]
\end{restatable}

\noindent We now discuss some background and motivation. Recall that $\FF(M)$ separates the points of $\lip(M)$ and the closed unit ball $B_{\lip(M)}$ is compact in the initial topology $\sigma(\lip(M), \FF(M))$, so $\FF(M)^\star$ is isometrically isomorphic to $\lip(M)$. The map $\delta_M:M\to \FF(M)$ is an isometric embedding, and $\FF(M)$ is a Banach space that linearises the Lipschitz structure of $M$: given any metric space $N$ with basepoint $0_N$ and any Lipschitz map $f:M\to N$ such that $f(0_M)=0_N$, there is a unique bounded linear map $\hat{f}:\mathcal{F}(M)\to \mathcal{F}(N)$ which extends $f$, and moreover $\lVert \hat{f} \rVert=\lVert f \rVert_{\Lip}$. Lipschitz-free spaces are related to the nonlinear theory of Banach spaces. In \cite{normedspaces}, it was proved that for any Banach space $X$ and net $N$ in $X$, $\lip(X)$ is isomorphic to a complemented subspace of $\lip(N)$ and $\lip(N)\sim \left (\bigoplus\limits_{n=1}^\infty \lip(N) \right )_{\ell_\infty}$. If the reverse complementation were proved, then by Pe\l{}czy\'{n}ski decomposition $\lip(N)\sim \lip(X)$. As observed in \cite{isolipschitzspace}, this would imply that uniformly homeomorphic Banach spaces have isomorphic spaces of Lipschitz functions. In \cite{Godefroy}, it was asked whether $\FF(\ell_1)$ is complemented in $\FF(\ell_1)^{\star\star}$, and it was noted that a positive answer would prove that any Banach space $X$ which is Lipschitz-isomorphic to $\ell_1$ is in fact linearly isomorphic to $\ell_1$. \\

\noindent Determining functional-analytic properties of $\FF(M)$ arising from metric properties of $M$ has revealed connections between Lipschitz-free spaces and both geometric measure theory and the theory of Lipschitz extensions. Relations between properties of a Lipschitz map $f$ and those of the linearisation $\hat{f}$ have also been studied. In \cite{LancienPernecka}, Lancien and Perneck\'a showed that for any doubling metric space $M$, $\FF(M)$ has the bounded approximation property. This depended on the existence of Lipschitz extension operators in doubling metric spaces, which was proved by Lee and Naor in \cite{LeeNaor}. In \cite{unrectifiable}, Aliaga, Gartland, Petitjean and Proch\'azka showed that for a complete metric space $M$, pure 1-unrectifiability of $M$ is equivalent to $\FF(M)$ having the Schur property, which is equivalent to $\FF(M)$ having the Radon-Nikod\'ym property. In \cite{curve flat}, Flores, Jung, Lancien, Petitjean,  Proch\'azka and Quilis proved that for a basepoint-preserving Lipschitz map $f:M\to N$ between pointed metric spaces with $M$ complete, curve-flatness of $f$ is equivalent to the linearised map $\hat{f}$ being a Radon-Nikod\'ym operator, which is equivalent to $\hat{f}$ being a Dunford-Pettis operator. \\

\noindent There has been much research on Lipschitz spaces and Lipschitz-free spaces over subsets of $\rr^n$, and some fundamental questions remain open. By Naor and Schechtman \cite{planarearthmover}, for $n\geq 2$, $\FF(\rr^n)$ is not isomorphic to a subspace of $\FF(\rr)$. Currently, it is unknown if $\FF(\rr^2)$ and $\FF(\rr^3)$ are isomorphic, and the questions of whether for $m\neq n$  both at least $2$,  $\FF(\rr^m)$ and $\FF(\rr^n)$ can be isomorphic, or if $\lip(\rr^m)$ and $\lip(\rr^n)$ can be isomorphic remain open. In \cite{isolipschitzspace} it was shown by Candido, C\'uth and Doucha that for every $d\in \mathbb{N}$, $\lip(\mathbb{Z}^d)$ and $\lip(\rr^d)$ are isomorphic. In \cite{Aliaga}, Aliaga generalised this and showed that if $M\subset \mathbb{R}^d$ is non-porous (in particular if $M$ is Lebesgue measurable with positive Lebesgue measure), then $\lip(M)$ is isomorphic to $\lip(\mathbb{R}^d)$. Aliaga notes that this implies that for any nonconstant real polynomial $p$, $\lip(\{(p(n), p(m)):n, m\in \mathbb{N}\})\sim \lip(\mathbb{Z}^2)$, and also proves that $\FF(\{(2^n, 2^m): n, m\in \mathbb{N}\})\sim \ell_1$. \cite{Aliaga}*{Question 6} asks whether the first isomorphism also holds for Lipschitz-free spaces. The $q=0$, $d=2$ case of Theorem \ref{thm1} implies that the answer is positive, and the $d=0$ case generalises the sparse sequences example. Theorem \ref{thm1} considers the natural generalisation to the case where both sub-exponentially growing sequences with consecutive ratios converging to $1$ and sparse sequences appear as factors. Note that the sequences $(\dindex{a}{n}{(i)})_{n=1}^\infty$ are allowed to grow much faster than polynomially - for example, $\exp\left(\frac{n}{\log\log n}\right)$ growth is permitted. \\

\subsection{Structure and overview}

\noindent The structure of this paper is as follows. In Section $2$, we collect several results from the literature that will be required for the proof of Theorem \ref{thm1}. In Section 3, we first prove the following key lemma. In the statement given below the notation $\dk$ denotes the truncated metric on $\rr^{d+q}$ given by the minimum of the $\ell_\infty$-norm and $k$. \\

\begin{restatable}{lemma}{lemtwo}
\label{lem2}
Let $\varepsilon>0$, $d\in \mathbb{N}$ and $q\in \mathbb{Z}_{\geq 0}$. Then there exists a constant $C=C_{\varepsilon, d, q}\geq 1$ depending only on $\varepsilon$, $d$ and $q$ such that the following hold.
\begin{enumerate}[label=(\roman*)]
    \item $\FF\left (\mathbb{Z}^d\times \{(1+\varepsilon)^m:m\in \mathbb{N}\}^q, \lVert \cdot \rVert_\infty\right )$ is $C$-isomorphic to a $C$-complemented subspace of $\FF\left(\mathbb{Z}^d, \lVert \cdot \rVert_\infty\right)$.
    \item $\forall k\in \mathbb{N}$,  $\FF\left (\mathbb{Z}^d\times \{(1+\varepsilon)^m:m\in \mathbb{N}\}^q, d^{\lVert \cdot \rVert_\infty}_k\right )$ is $C$-isomorphic to a $C$-complemented subspace of $\FF\left(\mathbb{Z}^d, \lVert \cdot \rVert_\infty \right)$.
\end{enumerate} 
\end{restatable}

\noindent We will then prove Theorem \ref{thm1}. We now give an overview of the proofs of Theorem \ref{thm1} in the case $d>0$ and Lemma \ref{lem2}. For Banach spaces $X$ and $Y$, the notation $X\comp Y$ denotes that $X$ is isomorphic to a complemented subspace of $Y$.
\begin{itemize}
    \item By Pe\l{}czy\'{n}ski decomposition, and using the known fact that $\FF(\mathbb{Z}^d)\sim \left (\bigoplus\limits_{n=1}^\infty\FF(\mathbb{Z}^d) \right )_{\ell_1}$, it suffices to prove that $\FF(\mathbb{Z}^d)\comp \FF(N)$ and $\FF(N)\comp \FF(\mathbb{Z}^d)$. To show the easier embedding $\FF(\mathbb{Z}^d)\comp\FF(N)$, letting $M=\{(\dindex{a}{n_1}{(1)}, \cdots, \dindex{a}{n_d}{(d)}):n_1, \cdots, n_d\in \mathbb{N}\}$, it suffices to show that $\FF(\mathbb{Z}^d)\comp \FF(M)$. To do this we find a sequence $(c_n)_{n=1}^\infty$ growing slightly faster than all of the $(\dindex{a}{n}{(i)})_{n=1}^\infty$ and containing a bilipschitz copy of each, which allows us to reduce to the case of all sequences being equal to $(c_n)_{n=1}^\infty$. By a result of Kalton from \cite{Kalton}, it will follow that $\FF(\mathbb{Z}^d)\comp \left (\bigoplus\limits_{k=1}^\infty  \FF\left(\{0, 1, \cdots, k\}^d\right)\right )_{\ell_1}$, so we only need to find a complemented copy of $\left (\bigoplus\limits_{k=1}^\infty  \FF\left(\{0, 1, \cdots, k\}^d\right)\right )_{\ell_1}$ in $\FF(M)$. The condition $\frac{c_{n+1}}{c_n}\to 1$ allows us to find a sequence $(A_k)_{k=1}^\infty$ of subsets of $\{c_n:n\in \mathbb{N}\}$ that are pairwise very far apart, each of which is uniformly bipilischitz equivalent to a $k+1$ term arithmetic progression. These subsets being at large distance from each other then implies that $\left (\bigoplus\limits_{k=1}^\infty  \FF(A_k^d)\right )_{\ell_1}\comp \FF(M)$, and the uniform equivalence with arithmetic progressions implies that the first space is isomorphic to  $\left (\bigoplus\limits_{k=1}^\infty  \FF\left(\{0, 1, \cdots, k\}^d\right)\right )_{\ell_1}$.
    \item The main ingredient in the proof of the harder embedding $\FF(N)\comp \FF(\mathbb{Z}^d)$ is Lemma \ref{lem2}, which is proved by induction on $q$, with the main argument being part $(ii)$ in the case $q=0$. The inductive step in the proof of Lemma \ref{lem2} is as follows: Write $M=\mathbb{Z}^d\times \{(1+\varepsilon)^m:m\in \mathbb{N}\}^{q+1}$ in the appropriate metric and let $A=\mathbb{Z}^d\times \{(1+\varepsilon)^m:m\in \mathbb{N}\}^q\times \{1+\varepsilon\}$. By a known quotient decomposition from \cite{Kaufmann}, it suffices to show the complementation for $\FF(A)$ and $\FF \left (\bigslant{M}{A} \right )$, where $\bigslant{M}{A}$ is the result of collapsing all of $A$ to a point. This will be defined in Section 2. The first of these complementations follows by the induction hypothesis. For the second, $\bigslant{M}{A}$ can be divided into slices corresponding to the final coordinate, which are sufficiently separated in $\bigslant{M}{A}$ that $\FF\left (\bigslant{M}{A} \right)$ can be written as the $\ell_1$ direct sum of the free spaces over these slices. For each slice we can apply the induction hypothesis, which concludes the inductive step since $\FF(\mathbb{Z}^d)\sim \left (\bigoplus\limits_{n=1}^\infty \FF(\mathbb{Z}^d) \right )_{\ell_1}$.
    \item For the $q=0$ case of Lemma \ref{lem2} part $(ii)$, we construct a $w^\star$-to-$w^\star$ continuous isomorphism between $\lip(\mathbb{Z}^d, d_k^{\lVert \cdot \rVert_\infty})$ and (approximately) an $\ell_\infty$ direct sum of Lipschitz spaces over $d$-dimensional discrete cubes of side length $k$ and their faces of all lower dimensions between $0$ and $d$. This is essentially done by restriction, with the caveat that at each dimension of faces we  use Lipschitz extension operators to ensure vanishing on the boundary of the appropriate cube face. This does not lose any information since we keep track of information on all dimensions of faces, and gives us the desired isomorphism. This passes to an isomorphism of preduals between $\FF(\mathbb{Z}^d, d_k^{\lVert \cdot \rVert_\infty})$ and (roughly) an $\ell_1$ direct sum of free spaces over these discrete cubes and faces. Each of these will be uniformly complemented in $\FF(\mathbb{Z}^d)$ which completes this case.
    \item Returning to the proof that $\FF(N)\comp \FF(\mathbb{Z}^d)$, by the aforementioned  result of Kalton (which will be stated in Section 2), $\FF(N)\comp\left ( \bigoplus\limits_{k=1}^\infty \FF(B_k) \right )_{\ell_1}$, where $(B_k)_{k=1}^\infty$ are concentric balls of radius $2^k$ in $N$. By taking $\varepsilon$ to be sufficiently small in Lemma \ref{lem2}, each $(\dindex{b}{m}{(j)})_{m=1}^\infty$ is bilipschitz equivalent to a subset of $\{(1+\varepsilon)^m:m\in \mathbb{N}\}$, and each $B_k$ is uniformly bilipschitz equivalent to a subset of $\mathbb{Z}^d\times \{(1+\varepsilon)^m:m\in \mathbb{N}\}^q$. This will imply by Lemma 2 that $\FF(B_k)$ is uniformly complemented in $\FF(\mathbb{Z}^d)$, which allows us to conclude. 
\end{itemize}

\section{Some results from the literature}

\noindent \textbf{Notation:} Recall that for Banach spaces $X$ and $Y$, we write $X \overset{c}{\hookrightarrow} Y$ to mean that $X$ is isomorphic to a complemented subspace of $Y$. Note that for Banach spaces $X, Y, Z$, if $X \overset{c}\hookrightarrow Y$ and $Y \overset{c}\hookrightarrow Z$, then $X \overset{c}\hookrightarrow Z$. For $C \geq 1$, we write $X \overset{C}{\sim} Y$ to denote that $X$ and $Y$ are $C$-isomorphic (i.e. if there exists $T:X\to Y$ a bounded linear bijection with $\lVert T \rVert\cdot \lVert T^{-1} \rVert\leq C$). Write $X\cong Y$ to denote that $X$ and $Y$ are isometrically isomorphic. For $C\geq 1$, a closed subspace $X$ of a Banach space $Y$ is $C$-complemented in $Y$ if there exists a bounded linear projection $P$ of $Y$ onto $X$ with $\lVert P \rVert\leq C$. We note that if $(X_n)_{n\geq 1}$, $(Y_n)_{n\geq 1}$ are sequences of Banach spaces such that for every $n$, $X_n$ is $C$-isomorphic to a $C$-complemented subspace of $Y_n$, then $\left (\bigoplus\limits_{n=1}^\infty X_n \right )_{\ell_1}$ is $C$-isomorphic to a $C$-complemented subspace of $\left ( \bigoplus\limits_{n=1}^\infty Y_n \right )_{\ell_1}$. 
For a metric space $(M, d)$, $x\in M$ and $r>0$, we write $\overline{B}(x, r)=\{y\in M:d(x, y)\leq r\}$ for the closed ball of radius $r$ in $M$ centred at $x$. 
For a metric space $M$, a basepoint $0\in M$ and a closed subset $N$ of $M$ with $0\in N$, we write $\Lip_N(M)$ for $\{f\in \lip(M):f\vert_N=0\}$, which is a closed subspace of $\lip(M)$ in the Lipschitz norm. We also write $\Lip(M)$ for the vector space of real valued Lipschitz functions on $M$. For a metric space $M$, if we take a basepoint $q\in M$ and an ambient point $0\in M$ already exists (for example if $M$ is a Banach space), we may sometimes write $\Lip_q(M)$ instead of $\lip(M)$ to avoid confusion.
\noindent For $k\in \mathbb{N}$, we write $\dk$ for the truncated metric on $\mathbb{R}^m$ (for any $m$) given by \[ \dk(p, q)=\min\{\lVert p-q\rVert_\infty, k\} \text{ for all } p, q\in \rr^m.\] For $n\in \mathbb{N}$ we write $[n]$ for $\{1, 2, \cdots, n\}$ and for $k\in \mathbb{N}_0$ we write $[n]^{(k)}$ for the family of subsets of $[n]$ of cardinality $k$. \\

\noindent The following is a separation condition that gives an $\ell_1$ direct sum decomposition for Lipschitz-free spaces.

\begin{definition} 
 \noindent Let $(M, d)$ be a metric space, $x_0 \in M$ and $\mathcal{C}=(C_i)_{i\in I}$ be an indexed family of subsets of $M$. \begin{itemize}
     \item For $0<\lambda\leq 1$, $\mathcal{C}$ is said to be \emph{well-separated with respect to $x_0$ with constant $\lambda$} if for for every $i,j\in I$ with $i\neq j$, every $x \in C_i$ and every $y \in C_j$, we have $d(x, y) \geq \lambda (d(x, x_0) + d(y, x_0))$. $\mathcal{C}$ is said to be \emph{well-separated with respect to $x_0$} if $\exists 0<\lambda\leq 1$ such that $\mathcal{C}$ is well-separated with respect to $x_0$ with constant $\lambda$.
     \item $\mathcal{C}$ is \emph{well-separated} if $\exists x_0 \in M$ such that $\mathcal{C}$ is well-separated with respect to $x_0$.
 \end{itemize} 
\end{definition}

\noindent The motivation for considering this notion of separation is that defining a Lipschitz function on the union becomes very simple: given an indexed family $(C_i)_{i\in I}$ that is well-separated with respect to $x_0$, defining a Lipschitz function on $\left (\bigcup\limits_{i\in I}C_i\right )\cup\{x_0\}$ that vanishes at $x_0$ is precisely the same as doing so on $C_i\cup\{x_0\}$ for each $i\in I$ in such a way that the Lipschitz constants are uniformly bounded. \\

\noindent The following lemma is a quantitative version of this and has been observed by various authors (for example the $p=1$ case of \cite{embeddability bases}*{Lemma 2.1}, \cite{remarksfreespace}*{Proposition 2}, or Proposition 5.1 from the arXiv version of \cite{Kaufmann}).

 \begin{lemma}\label{well separated} Let $(M, d)$ be a metric space, $x_0\in M$, $0<\lambda\leq 1$ and $(M_n)_{n\geq 1}$ be a sequence of subsets of $M$ that are well-separated with respect to $x_0$ with constant $\lambda$. Then
 \[\mathcal{F}\left (\bigcup\limits_{n=1}^\infty M_n\cup\{x_0\} \right )\overset{\lambda^{-1}}{\sim}\left (\bigoplus\limits_{n=1}^\infty \mathcal{F}(M_n\cup\{x_0\}) \right )_{\ell_1}.\]  
 \end{lemma}

 \noindent The next lemma may be found as the $p=1$ case of \cite{Lipschitzsumgeometric}*{Lemma 2.8 $(ii)$, $(iii)$}. 
\begin{lemma}\label{remove a point}
     There exists a universal constant $C\geq 1$ such that the following hold.
     \begin{enumerate}[label=(\roman*)]
     \item For every metric space $M$ with $\vert M \vert\geq 2$ and every $x_0\in M$, $\FF(M)\overset{C}{\sim} \FF(M\setminus\{x_0\})\oplus_1 \mathbb{R}$.
         \item For every infinite metric space $M$ and every $x_0\in M$, we have that $\FF(M)\overset{C}{\sim}\FF(M\setminus\{x_0\})$. \\
     \end{enumerate}
\end{lemma}

\noindent The following result states that every infinite-dimensional Lipschitz-free space contains a complemented copy of $\ell_1$.

\begin{proposition}[\cite{ell1comp}*{Theorem 1 (i)}]\label{l1}
    For any infinite metric space $M$, $\ell_1\comp \FF(M)$.
\end{proposition}

 \noindent The following lemma is an application of Lemmas \ref{well separated} and \ref{remove a point}.

 \begin{lemma}[\cite{Aliaga}*{Lemma 3.3}]\label{well separated disjoint}
    \noindent Let $M$ be a metric space and $(M_n)_{n=1}^\infty$ be a sequence of non-empty, pairwise disjoint, well-separated subsets of $M$. Then $\mathcal{F}\left (\bigcup\limits_{n=1}^\infty M_n\right ) \sim \left (\bigoplus\limits_{n=1}^\infty \mathcal{F}(M_n)\right )_{\ell_1}\oplus_1 \ell_1$. \\
\end{lemma} 
\noindent We remark that if $\vert M_n \vert\geq 2$ for infinitely many $n$ or if some $M_n$ is infinite, then $\ell_1\comp \left (\bigoplus\limits_{n=1}^\infty\FF(M_n) \right )_{\ell_1}$. So the $\ell_1$ factor in Lemma \ref{well separated disjoint} is not required unless all $M_n$ are finite and all but finitely many $M_n$ are singletons. \\

\noindent Before the next result, we briefly recall the definition of a doubling metric space.
\begin{definition}
     A metric space $M$ is \emph{doubling} if $\exists k\in \mathbb{N}$ such that $\forall r>0$, $\forall x\in M$, $\exists x_1, \cdots, x_k\in M$ with $\overline{B}(x, r)\subset \bigcup\limits_{i=1}^k \overline{B}(x_i, \frac{r}{2})$ (i.e. every closed ball can be covered by $k$ closed balls of half its radius). The least such $k$ is denoted by $D(M)$ and is called the \emph{doubling constant} of $M$. 
\end{definition}
\noindent The construction underlying the following result is the existence of gentle partitions of unity in doubling metric spaces, which was proved by Lee and Naor in \cite{LeeNaor}. It follows immediately from \cite{isolipschitzspace}*{Proposition 1.8}, using the fact that any subset of a doubling metric space is doubling. We note that the result holds more generally: \cite{isolipschitzspace} only assumes that $N$ is doubling.
\begin{proposition}[\cite{isolipschitzspace}*{Proposition 1.8}]\label{Complementation} Let $N$ and $M$ be non-empty subsets of a doubling metric space with $N\subset M$. Then $\mathcal{F}(N) \overset{c}{\hookrightarrow} \mathcal{F}(M)$ and $\lip(N) \overset{c}\hookrightarrow \lip(M)$. 
\end{proposition} 

\noindent The next decomposition result is due to Kalton, but in a form slightly different from the original, which may be found as Proposition 4.3 of \cite{Kalton}. In the original, the direct sum runs over $\mathbb{Z}$ and there is no assertion of complementation. However, this version follows from Kalton's construction, see \cite{normedspaces}*{Lemma 1.2}.

\begin{lemma}\label{Kalton decomposition} Let $M$ be a metric space and $0 \in M$ be a basepoint. Then
\[
\mathcal{F}(M) \overset{c}\hookrightarrow \left( \bigoplus_{k=1}^{\infty} \mathcal{F}(\overline{B}(0, 2^k)) \right)_{\ell_1}.
\]
\end{lemma}

\noindent The next result is a special instance of \cite{Lipschitzsumgeometric}*{Corollary 5.10} in the case $p=1$ , $M=\mathbb{R}^d$ and $N=\mathbb{Z}^d$, with the metric on $\mathbb{R}^d$ being that derived from the $\lVert \cdot \rVert_\infty$ norm. Note that $\mathbb{Z}^d$ may also be replaced more generally by a net $N$ in any Banach space $X$ (see \cite{normedspaces}*{Theorem 3.6}).

\begin{proposition}\label{Z^d l_1 sum} For $d \in \mathbb{N}$, $\mathcal{F}\left(\mathbb{Z}^d, \lVert \cdot \rVert_\infty\right) \sim \left( \bigoplus\limits_{k=1}^{\infty} \mathcal{F}\left(\mathbb{Z}^d, \lVert \cdot \rVert_\infty\right) \right)_{\ell_1}$ .
\end{proposition}

\noindent We recall some standard terminology in Definitions \ref{pointwisetop} and \ref{ext}.

 \begin{definition}\label{pointwisetop} For a metric space $M$, the \emph{pointwise topology} on $\lip(M)$ is the initial topology generated by the collection of evaluation maps $\{\delta_M(x): x\in M\}$. 
 \end{definition}

 \noindent We recall that for a metric space $M$, on each norm bounded subset of $\lip(M)$,  the $w^\star$-topology induced by the canonical identification $\FF(M)^\star \cong \lip(M)$ coincides with the pointwise topology.

 \begin{definition}\label{ext} Let $M$ be a metric space, $0\in M$ be a basepoint and $N$ be a subset of $M$ containing $0$.  
 \begin{itemize}
     \item $\Ext_0(N, M)$ is defined to be the collection of bounded linear maps $T:\lip(N)\to \lip(M)$ such that $\forall f\in \lip(N)$, $\left (Tf \right )\vert_N=f$. 
     \item $\Ext_0^{pt}(N, M)$ is defined to be the collection of all $T\in \Ext_0(N, M)$ that are $w^\star$-to-$w^\star$ continuous. Equivalently, $\Ext_0^{pt}(N, M)$ is the collection of all $T\in \Ext_0(N, M)$ such that $T\vert_{B_{\lip(N)}}$ is pointwise to pointwise continuous.
 \end{itemize}
 \noindent In some cases, when a basepoint $q\in N$ is taken and an ambient point $0$ exists in $M$ (for example if $M$ is a Banach space), we may write $\Ext_q(N, M)$ or $\Ext_q^{pt}(N, M)$ to avoid confusion.
 \end{definition}

 \noindent The equivalence in the second bullet point comes from the Banach-Dieudonn\'e theorem, which implies for Banach spaces $X, Y$ and $T\in \mathcal{B}(X^\star, Y^\star)$, that $T$ is $w^\star$-to-$w^\star$ continuous iff $T\vert_{B_{X^\star}}$ is $w^\star$-to-$w^\star$ continuous. Before the next lemma we recall the following notion of a quotient metric space.
 \begin{definition}
      Let $(M, d)$ be a metric space and $A$ be a non-empty closed subset of $M$. $\bigslant{M}{A}:=(M\setminus A)\cup\{A\}$ is a metric space in the quotient metric $d_{\bigslant{M}{A}}$ defined by
    $d_{\bigslant{M}{A}}(m, A)=d(m, A)$ for all $m\in M\setminus A$, and $d_{\bigslant{M}{A}}(m, n)=\min\{d(m, n), d(m, A)+d(n, A)\}$ for all $m, n\in M\setminus A$. 
 \end{definition}
 \noindent This can be thought of as the result of collapsing $A$ to a point. \\
 
 \noindent The existence of $w^\star$-to-$w^\star$ continuous extension operators implies the following decomposition result relating the free space of a metric space $M$ to those of $\bigslant{M}{A}$ and $A$ for a closed subset $A$ of $M$. It can informally be thought of as splitting a Lipschitz function on $M$ into a Lipschitz function on $A$ and a Lipschitz function on $M$ which vanishes on $A$. 

 \begin{lemma}[\cite{Kaufmann}*{Lemma 2.2}]\label{Quotient decomposition} \ Let $(M, d)$ be a metric space, $0\in M$ be a basepoint, and $A$ be a closed subset of $M$ containing $0$. Assume that $T\in \Ext_0^{pt}(A, M)$. Then $\FF(M)$ and $\FF\left (\bigslant{M}{A} \right )\oplus_1 \FF(A)$ are $(1+\lVert T \rVert)^2$-isomorphic. 
 \end{lemma}
  \noindent\textbf{Note:} If $r:M\to A$ is an $L$-Lipschitz retraction, then $f\mapsto f\circ r:\lip(A)\to \lip(M)$ defines an element of $\Ext_0^{pt}(A, M)$ with norm at most $L$. Hence Lemma \ref{Quotient decomposition} applies and we obtain $\FF(M)\overset{(1+L)^2}{\sim}\FF\left (\bigslant{M}{A} \right )\oplus_1 \FF(A)$. \\

\noindent We will use this lemma to prove a quantitative version of Proposition \ref{Complementation} (see Proposition \ref{quotientquantitative} below). We first recall the following consequence of the work of Lee and Naor in \cite{LeeNaor}, which may be found in \cite{isolipschitzspace} as Proposition 1.8. Note that \cite{isolipschitzspace} only requires that $N$ is doubling.

\begin{proposition}\label{Doubling extension} There exists a universal constant $C>0$ with the following property. Let $M$ be a metric space, $0\in M$ be a basepoint, and $N$ be a doubling subset of $M$ containing $0$. Then there exists $T\in \Ext_0^{pt}(N, M)$ with $\lVert T \rVert\leq C(1+\log D(N))$ . 
\end{proposition}
 
 \noindent Since $(\mathbb{R}^d, \lVert \cdot \rVert_\infty)$ is doubling and for any doubling metric space $M$ and $N\subset M$ we have that $D(N)\leq D(M)^2$, combining Proposition \ref{Doubling extension} with Lemma \ref{Quotient decomposition}, we obtain: \\

\begin{proposition}
    \label{quotientquantitative} Let $d\in \mathbb{N}$, and take the metric on $\mathbb{R}^d$ and all of its subsets to be that derived from the $\lVert \cdot \rVert_\infty$ norm. Then there exists a constant $C_d\geq 1$ such that for any $\emptyset\neq N\subset M\subset \mathbb{R}^d$ with $N$ closed in $M$ and any basepoint $q\in N$, we have that the following hold.
 \begin{itemize}
     \item $\exists  T\in \Ext_q^{pt}(N, M)$ with $\lVert T \rVert\leq C_d$, and such that $\forall f\in \mathrm{Lip}_q(N)$, $\lVert Tf \rVert_{\infty}=\lVert f \rVert_{\infty}$.
     \item $\FF(M)\overset{C_d}{\sim}\FF\left(\bigslant{M}{N}\right)\oplus_1 \FF(N)$, and hence both of $\FF \left (\bigslant{M}{N} \right )$ and $\FF(N)$ are $C_d$-isomorphic to a $C_d$-complemented subspace of $\FF(M)$.
 \end{itemize} 
\end{proposition}

\begin{proof}This follows apart from the claim on the $\lVert \cdot \rVert_{\infty}$ norms. To see this, note that the proof of Proposition 1.8 in \cite{isolipschitzspace} shows that $M$ admits a $C_d$-gentle partition of unity with respect to $N$. Then in the reference given in \cite{isolipschitzspace} to \cite{LancienPernecka}*{page 2325}, note that the extension given in terms of an integral satisfies this requirement since in their notation, $\lVert \psi_x \rVert_1=1$ for any $x\in M\setminus N$.
\end{proof}

 \section{Proofs of Lemma 2 and Theorem 1}

 \noindent We first make the following observation, the proof of which is easy and is therefore omitted. 

 \begin{observation}\label{lipN as quotient}
      Let $M$ be a metric space, $0\in M$ be a basepoint and $N\subset M$ be a closed subset containing $0$. Take $\bigslant{M}{N}$ to have basepoint $N$. Then the map
 $f\mapsto \tilde{f}:\Lip_N(M)\to \lip\left (\bigslant{M}{N} \right )$, where
 \[\tilde{f}(p)=\begin{cases} f(p) & \text{ if }p\in M\setminus N \\ 0 &\text{ if } p=N
 \end{cases}\]
 \noindent is an isometric isomorphism. 
 \end{observation}

 \noindent We now prove Lemma \ref{lem2}, which we restate here for convenience. Part $(ii)$ is not used in the proof of Theorem \ref{thm1} but is used to make the inductive step in the proof of part $(i)$ go through. \\

 \lemtwo*

 \begin{proof} We proceed by induction on $q$. Throughout, given quantities $\alpha, \beta, \cdots, \gamma$, the notation $C_{\alpha, \beta, \cdots, \gamma}$ will denote a constant depending only on $\alpha, \beta, \cdots, \gamma$. The value of $C_{\alpha, \beta, \cdots, \gamma}$ will be allowed to change from line to line, provided that its dependencies do not. Unless otherwise specified, the metric on $\mathbb{R}^{d+q}$ and all of its subsets is taken to be the metric derived from the $\lVert \cdot \rVert_\infty$-norm. \\

 \noindent\textbf{1: Proof in the case $q=0$} \\

 \noindent In this case part $(i)$ is obvious, and it remains to consider, for $k\in \mathbb{N}$, the space $\FF\left (\mathbb{Z}^d, d^{\lVert \cdot \rVert_\infty}_k\right )$. 
 Fix $k\in \mathbb{N}$. We construct an isomorphism between $\lip\left(\mathbb{Z}^d, \dk\right)$ and an $\ell_\infty$ direct sum of Lipschitz spaces over quotients of discrete cubes. We will then show that this map is $w^\star$-to-$w^\star$ continuous, which will allow us to pass to the preduals and obtain an isomorphism of $\FF\left(\mathbb{Z}^d, \dk\right)$ with an $\ell_1$ direct sum of Lipschitz-free spaces over quotients of discrete cubes. Finally, each of the Lipschitz-free spaces in this $\ell_1$ direct sum will be uniformly complemented in $\FF\left(\mathbb{Z}^d, \lVert \cdot \rVert_\infty\right)$ by Proposition \ref{quotientquantitative}, which will give the result.  \\
 
 \noindent We introduce the following notation for discrete cubes of dimensions between $0$ and $d$, with side length $k$, and for their discrete boundaries:
 for $\mm\in \mathbb{Z}^d$, $0\leq j\leq d$, $A\in [d]^{(j)}$ and $\varepsilon\in \{0, 1\}^A$, set \[C_{\mm, \varepsilon, A}^{d-j}:=B_1\times \cdots \times B_d,\] where
 \[B_i=\begin{cases}
     \{k(m_i+\varepsilon_i)\} & \text{if } i\in A \\
     \{km_i, km_i+1, \cdots, k(m_i+1)\} & \text{if } i\not\in A.
 \end{cases}\]
\noindent This is a $d-j$ dimensional discrete cube of side length $k$. Set $C_{\mm, A}^{d-j}:=C_{\mm, 0_A, A}^{d-j}$, where $0_A:A\to \{0, 1\}$ is the zero function. Recall that $[d]^{(j)}$ is the family of cardinality $j$ subsets of $\{1, 2, \cdots, d\}$. \\

\noindent  For $0\leq j\leq d-1$, $\mm\in \mathbb{Z}^d$, $A\in [d]^{(j)}$, also set 
\[\partial C_{\mm, A}^{d-j}:=\bigcup\limits_{i\in [d]\setminus A} \  \bigcup\limits_{\substack{\varepsilon\in \{0, 1\}^{A\cup\{i\}} \\ \varepsilon_a=0 \  \forall a\in A}} C_{\mm, \varepsilon, A\cup\{i\}}^{d-j-1},\] the discrete boundary of $C_{\mm, A}^{d-j}$. \\
Note that for any $0\leq j\leq d$, any $\mm, \nn\in \mathbb{Z}^d$ and any $A, B\in [d]^{(j)}$, we have that $C_{\mm, A}^{d-j}=C_{\nn, B}^{d-j} \Leftrightarrow (\mm, A)=(\nn, B)$. 
Note also that for any $0\leq j\leq d-1$, any $\mm,\nn\in \mathbb{Z}^d$ and any $A, B\in [d]^{(j)}$ with $(\mm, A)\neq (\nn, B)$, we have that
\[C_{\mm, A}^{d-j}\cap C_{\nn, B}^{d-j}\subset \bigcup\limits_{\substack{\pp\in \mathbb{Z}^d \\ D\in [d]^{(j)}}} \partial C_{\pp, D}^{d-j}=\bigcup\limits_{\substack{\qq\in \mathbb{Z}^d \\ E\in [d]^{(j+1)}}}C_{\qq, E}^{d-j-1}.\]

\noindent \textbf{1.1: Construction of the isomorphism with an $\ell_\infty$ direct sum of Lipschitz spaces over quotients of cubes} \\

\noindent Let $\mm\in \mathbb{Z}^d$, $0\leq j\leq d-1$ and $A\in [d]^{(j)}$.
By Proposition \ref{quotientquantitative}, we can fix $S_{\mm, A}^{d-j}\in \Ext_{k\mm}^{pt}\left(\partial C_{\mm, A}^{d-j}, C_{\mm, A}^{d-j}\right)$ such that $\lVert S_{\mm, A}^{d-j} \rVert\leq C_d$, and such that $\forall g\in \Lip_{k\mm}\left(\partial C_{\mm, A}^{d-j}\right)$, $\lVert S_{\mm, A}^{d-j}g\rVert_{\infty}=\lVert g \rVert_\infty$.
Here the base point of both $\partial C_{\mm, A}^{d-j}$ and $C_{\mm, A}^{d-j}$ is taken to be $k\mm$, which is reflected in the notation. 
Now let $T_{\mm, A}^{d-j}:\Lip\left(\partial C_{\mm, A}^{d-j}\right)\to \Lip\left(C_{\mm, A}^{d-j}\right)$ be defined by 
\[T_{\mm, A}^{d-j}g=S_{\mm, A}^{d-j}(g-g(k\mm))+g(k\mm) \qquad \text{ for all }g\in \Lip\left(\partial C_{\mm, A}^{d-j}\right).\] 
Then we have that \[\refstepcounter{equation}\label{dagger} \forall g\in \Lip\left(\partial C_{\mm, A}^{d-j}\right), \   \left (T_{\mm, A}^{d-j}g\right )\bigg\vert_{\partial C_{\mm, A}^{d-j}}=g, \  \lVert T_{\mm, A}^{d-j} g \rVert_{\Lip}\leq C_d \lVert g\rVert_{\Lip}, \text{ and } \lVert T_{\mm, A}^{d-j} g \rVert_\infty\leq 3\lVert g \rVert_\infty,  \tag{\theequation}\] with the last claim following by the triangle inequality.
\noindent Note that as $\Lip\left(\partial C_{\mm, A}^{d-j}\right)$ and $\Lip\left(C_{\mm, A}^{d-j}\right)$ are finite-dimensional, at each $z\in C_{\mm, A}^{d-j}$, $\left ( T_{\mm, A}^{d-j}g\right )(z)$ is given by some fixed linear combination of the values $\left (g(y) \right )_{y\in \partial C_{\mm, A}^{d-j}}$ that is independent of $g$, i.e. for any $z\in C_{\mm, A}^{d-j}$, there exists $\left(\alpha_{y, z}\right)_{y\in \partial C_{\mm, A}^{d-j}}\subset \rr$ such that
\[\refstepcounter{equation}\label{daggerdagger}\forall g\in \Lip\left(\partial C_{\mm, A}^{d-j}\right ), \ \left (T_{\mm, A}^{d-j}g \right )(z)=\sum\limits_{y\in \partial C_{\mm, A}^{d-j}} \alpha_{y, z} g(y). \tag{\theequation}\]

\noindent Suppose $f\in \lip\left(\mathbb{Z}^d, \dk\right)$ (with basepoint taken to be $0$), $0\leq j\leq d-1$, $\mm\in \mathbb{Z}^d$ and $A\in [d]^{(j)}$. 
Then $f\vert_{C_{\mm, A}^{d-j }}-T_{\mm, A}^{d-j}\left (f\vert_{\partial C_{\mm, A}^{d-j}}\right ):C_{\mm, A}^{d-j}\to \mathbb{R}$ is Lipschitz with Lipschitz constant $\leq(1+ C_d) \lVert f \rVert_{\Lip\left(\mathbb{Z}^d, \dk\right)}$. 
Note that $f\vert_{C_{\mm, A}^{d-j}}-T_{\mm, A}^{d-j}\left(f\vert_{\partial C_{\mm, A}^{d-j}}\right)$ vanishes on $\partial C_{\mm, A}^{d-j}$, hence is in $\Lip_{\partial C_{\mm, A}^{d-j}}\left(C_{\mm, A}^{d-j}\right)$. 
Note also that for any $\mm\in \mathbb{Z}^d$, 
\[\vert f(k\mm) \vert=\vert f(k\mm)-f(0) \vert\leq \lVert f \rVert_{\Lip\left(\mathbb{Z}^d, \dk\right)}\dk(k\mm, 0)\leq k\lVert f \rVert_{\Lip\left(\mathbb{Z}^d, \dk\right)}.\]
\noindent Letting $W=\{a\in \ell_\infty(\mathbb{Z}^d):a_{\underline{0}}=0\}$ in the $\ell_\infty$ norm, we therefore have that
\begin{align*}
    & T:\lip\left(\mathbb{Z}^d, \dk\right)\to X:=\left (\bigoplus\limits_{\substack{\mm\in \mathbb{Z}^d \\ 0\leq j\leq d-1 \\ A\in [d]^{(j)}}} \Lip_{\partial C_{\mm, A}^{d-j}}\left(C_{\mm, A}^{d-j}\right) \right )_{\ell_\infty} \oplus_\infty W \\
    & f\mapsto \left ( \left (f\vert_{C_{\mm, A}^{d-j}}-T_{\mm, A}^{d-j}\left(f\vert_{\partial C_{\mm, A}^{d-j}}\right)\right )_{\substack{\mm\in \mathbb{Z}^d \\ 0\leq j\leq d-1 \\ A\in [d]^{(j)}}}, \left (\frac{f(k\mm)}{k} \right )_{\mm\in \mathbb{Z}^d}\right )
\end{align*}
\noindent is bounded and linear with $\lVert T \rVert\leq 1+C_d$. \\

\noindent\textbf{1.2: Norm bound on the inverse map} \\

\noindent We construct $S:X\to \lip\left(\mathbb{Z}^d, \dk\right)$ inverse to $T$ and bound its norm. \\
Let $x=\left ((h_{\mm, A}^{d-j})_{\substack{\mm\in \mathbb{Z}^d, \\ 0\leq j\leq d-1 \\ A\in [d]^{(j)}}}, (a_{\mm})_{\mm\in \mathbb{Z}^d} \right )\in X$. 
Let $C^0=k\mathbb{Z}^d$, the union of the $0$-dimensional discrete cubes, and define $j^0:C^0\to \mathbb{R}$ by $j^0(\nn)=ka_{\frac{\nn}{k}}$. For $\mm\in \mathbb{Z}^d$ and $A\in [d]^{(d-1)}$, let $j_{\mm, A}^0=j^0\vert_{\partial C_{\mm,A}^1}$.
Note that $\lVert j_{\mm, A}^0 \rVert_\infty\leq k\lVert a \rVert_\infty\leq k \lVert x \rVert$ and $\partial C_{\mm, A}^1$ consists of two points at $\lVert \cdot \rVert_\infty$-distance $k$, so by the triangle inequality, $\lVert j_{\mm, A}^0\rVert_{\Lip(\partial C_{\mm, A}^1, \lVert \cdot \rVert_\infty)}\leq  2\lVert x \rVert$. 
Then $j_{\mm, A}^1=h_{\mm, A}^1+T_{\mm, A}^1 j_{\mm, A}^0: C_{\mm, A}^1\to \mathbb{R}$ satisfies $\lVert j_{\mm, A}^1 \rVert_{\Lip(C_{\mm, A}^1, \lVert \cdot \rVert_{\infty})}\leq C_d \lVert x \rVert$ (recall that $C_d$ is allowed to change from line to line). 
Also, $\lVert h_{\mm, A}^1 \rVert_\infty\leq k\lVert x \rVert$ as $h_{\mm, A}^1=0$ on $\partial C_{\mm, A}^1$ and $\mathrm{diam}(C_{\mm, A}^1)\leq k$. 
We have that $\lVert T_{\mm, A}^1 j_{\mm, A}^0\rVert_\infty\leq 3\lVert j_{\mm, A}^0 \rVert_\infty$ by \eqcref{dagger}, so $\lVert j_{\mm, A}^1 \rVert_\infty\leq kC_d \lVert x\rVert$. \\

\noindent On $\partial C_{\mm, A}^1$, $j_{\mm, A}^1$ agrees with $j_{\mm, A}^0=j^0\vert_{\partial C_{\mm, A}^1}$. 
Thus we may define \[j^1: C^1=\bigcup\limits_{\substack{\mm\in \mathbb{Z}^d \\ A\in [d]^{(d-1)}}}C_{\mm, A}^1\to \mathbb{R}\] by setting $j^1$ to be equal to $j_{\mm, A}^1$ on $C_{\mm, A}^1$ for each $(\mm, A)$. 
Then  $j^1\vert_{C^0}=j^0$ and $\lVert j^1 \rVert_\infty\leq kC_d \lVert x \rVert$. 
Since for each $(\mm, A)$ we have that $\lVert j_{\mm, A}^1 \rVert_{\Lip(C_{\mm, A}^1, \lVert \cdot \rVert_{\infty})}\leq C_d \lVert x \rVert$, by redefining $C_d$ we obtain that $\lVert j^1 \rVert_{\Lip(C^1, \lVert \cdot \rVert_\infty)}\leq C_d \lVert x \rVert$. \\

\noindent Now if $1\leq s\leq d-1$ and $j^s:C^s=\bigcup\limits_{\substack{\mm\in \mathbb{Z}^d \\ A\in [d]^{(d-s)}}}C_{\mm, A}^s=\bigcup\limits_{\substack{\nn\in \mathbb{Z}^d \\ B\in [d]^{(d-s-1)}}}\partial C_{\nn, B}^{s+1}\to \mathbb{R}$ has been defined with $\lVert j^s \rVert_{\Lip(C^s, \lVert \cdot \rVert_\infty)}\leq C_d \lVert x \rVert$ and $\lVert j^s \rVert_\infty\leq kC_d \lVert x \rVert$, define $j^{s+1}$ as follows: 
for $\mm\in \mathbb{Z}^d$, $A\in [d]^{(d-s-1)}$, note that $j_{\mm, A}^{s+1}:=h_{\mm, A}^{s+1}+T_{\mm, A}^{s+1}\left ( j^s \vert_{\partial C_{\mm, A}^{s+1}}\right ):C_{\mm, A}^{s+1}\to \mathbb{R}$ is well-defined since  $\partial C_{\mm, A}^{s+1}$ is contained in $C^s$. 
We have that $\lVert j_{\mm, A}^{s+1} \rVert_{\Lip(C_{\mm, A}^{s+1}, \lVert \cdot \rVert_\infty)}\leq C_d \lVert x \rVert$ by \mbox{\eqcref{dagger}} and the triangle inequality. 
Also, $\lVert h_{\mm, A}^{s+1} \rVert_\infty\leq k \lVert x \rVert$ since $\mathrm{diam}(C_{\mm, A}^{s+1})=k$ and $h_{\mm, A}^{s+1}=0$ on $\partial C_{\mm, A}^{s+1}$. 
So again by \mbox{\eqcref{dagger}}, $\lVert j_{\mm, A}^{s+1} \rVert_\infty\leq kC_d \lVert x \rVert$. 
Note that the restriction of $j_{\mm, A}^{s+1}$ to $\partial C_{\mm, A}^{s+1}$ agrees with that of $j^s$. 
Thus we may define $j^{s+1}:C^{s+1}=\bigcup\limits_{\substack{\mm\in \mathbb{Z}^d \\ A\in [d]^{(d-s-1)}}}C_{\mm, A}^{s+1}\to \mathbb{R}$ by setting $j^{s+1}$ to be equal to $j_{\mm, A}^{s+1}$ on $C_{\mm, A}^{s+1}$ for each $(\mm, A)$. 
Then $\lVert j^{s+1} \rVert_{\Lip(C^{s+1}, \lVert \cdot \rVert_\infty)}\leq C_d \lVert x \rVert$, $\lVert j^{s+1} \rVert_\infty\leq kC_d \lVert x \rVert$ (recall that $C_d$ is allowed to change from line to line), and $j^{s+1}$ extends $j^s$. \\

\noindent Note that $C^d=\mathbb{Z}^d$, and so we set $Sx=j^{d}:\mathbb{Z}^d\to \mathbb{R}$. Since the inductive construction above is carried out $d$ times, we obtain that $\lVert j^d \rVert_\infty \leq kC_d \lVert x \rVert$ and $\lVert j^d \rVert_{\Lip(\mathbb{Z}^d, \lVert \cdot \rVert_\infty)}\leq C_d \lVert x \rVert$. 
We have $j^d(0)=0$ since $a_{\underline{0}}=0$, and by the previous two inequalities for $j^d$, $\lVert j^d \rVert_{\Lip\left(\mathbb{Z}^d, \dk\right)}\leq C_d \lVert x \rVert$. 
Then $S:X\to \lip\left(\mathbb{Z}^d, \dk\right)$ is bounded and linear with $\lVert S \rVert\leq C_d$, and we observe that $S$ and $T$ are inverse maps. 
Combining this with Observation \ref{lipN as quotient}, we obtain that (where the base point of each $\bigslant{C_{\mm, A}^{d-j}}{\partial C_{\mm, A}^{d-j}}$ is taken to be $\partial C_{\mm, A}^{d-j}$)
\[\refstepcounter{equation}\label{sum}\lip\left(\mathbb{Z}^d, \dk\right)\overset{C_d}{\sim} \left (\bigoplus\limits_{\substack{\mm\in \mathbb{Z}^d \\ 0\leq j\leq d-1 \\ A\in [d]^{(j)}}} \lip\left(\bigslant{C_{\mm, A}^{d-j}}{\partial C_{\mm, A}^{d-j}}\right) \right )_{\ell_\infty} \oplus_\infty \ell_\infty (\mathbb{Z}^d\setminus \{0\}) \tag{\theequation}\]
via
\[f\mapsto \left ( \left (\reallywidetilde{f\vert_{C_{\mm, A}^{d-j}}-T_{\mm, A}^{d-j}\left(f\vert_{\partial C_{\mm, A}^{d-j}}\right)}\right )_{\substack{\mm\in \mathbb{Z}^d \\ 0\leq j\leq d-1 \\ A\in [d]^{(j)}}}, \left (\frac{f(k\mm)}{k} \right )_{\mm\in \mathbb{Z}^d\setminus\{0\}}\right ).\]

\noindent\textbf{1.3: Proof of $w^\star$-to-$w^\star$-continuity and conclusion of the $q=0$ case} \\

\noindent The second space in \mbox{\eqcref{sum}} isometrically the dual of $Y:=\left (\bigoplus\limits_{\substack{\mm\in \mathbb{Z}^d \\ 0\leq j\leq d-1 \\ A\in [d]^{(j)}}} \FF\left(\bigslant{C_{\mm, A}^{d-j}}{\partial C_{\mm, A}^{d-j}}\right) \right )_{\ell_1} \oplus_1 \ell_1 (\mathbb{Z}^d\setminus \{0\})$ (with the same choices of basepoints), and to pass to an isomorphism of $\FF\left(\mathbb{Z}^d, \dk\right)$ and $Y$, we need to show that the map above is $w^\star$-to-$w^\star$ continuous.
It suffices to show that the restriction to $B_{\lip\left(\mathbb{Z}^d, \dk\right)}$ is pointwise-to-$w^\star$ continuous. 
\noindent Given a sequence $(X_n)_{n=1}^\infty$ of Banach spaces, with the usual isometric identification of $\left (\bigoplus\limits_{n=1}^\infty X_n \right )_{\ell_1}^\star$ with $\left (\bigoplus\limits_{n=1}^\infty X_n^\star \right )_{\ell_\infty}$, on any norm bounded subset, the restriction of the weak star topology of $\left (\bigoplus\limits_{n=1}^\infty X_n^\star \right )_{\ell_\infty}$ agrees with the restriction of the product topology $\prod\limits_{n=1}^\infty (X_n^\star, w^\star)$. 
Under the usual isometric identification $\ell_\infty(\mathbb{Z}^d\setminus\{0\})\cong \ell_1(\mathbb{Z}^d\setminus\{0\})^{\star}$, the weak star topology on any norm bounded subset of $\ell_\infty(\mathbb{Z}^d\setminus\{0\})$ agrees with the topology of pointwise convergence, as does the weak star topology on any norm bounded subset of each $\lip\left(\bigslant{C_{\mm, A}^{d-j}}{\partial C_{\mm, A}^{d-j}}\right) $ (as each is finite-dimensional). 
It therefore suffices to show that the restriction of each coordinate of the map above to $B_{\lip\left(\mathbb{Z}^d, \dk\right)}$ is pointwise-to-pointwise continuous. 
This holds by \mbox{\eqcref{daggerdagger}} and therefore we obtain
\[\FF\left(\mathbb{Z}^d, \dk\right)\overset{C_d}{\sim} \left (\bigoplus\limits_{\substack{\mm\in \mathbb{Z}^d \\ 0\leq j\leq d-1 \\ A\in [d]^{(j)}}} \FF\left(\bigslant{C_{\mm, A}^{d-j}}{\partial C_{\mm, A}^{d-j}}\right) \right )_{\ell_1} \oplus_1 \ell_1 (\mathbb{Z}^d\setminus \{0\}).\]
\noindent By Proposition \ref{quotientquantitative}, each $\FF \left (\bigslant{C_{\mm, A}^{d-j}}{\partial C_{\mm, A}^{d-j}} \right )$ is $C_d$-isomorphic to a $C_d$-complemented subspace of $\FF\left (C_{\mm, A}^{d-j}, \lVert \cdot \rVert_\infty \right )$. 
Using that $\ell_1(\mathbb{Z}^d\setminus\{0\})\cong\ell_1\cong \FF(\mathbb{Z})\comp \FF(\mathbb{Z}^d, \lVert \cdot \rVert_\infty)$ by Proposition \ref{Complementation}, we deduce that $\FF\left(\mathbb{Z}^d, \dk\right)$ is $C_d$-isomorphic to a $C_d$-complemented subspace of 
\[\left (\bigoplus\limits_{\substack{\mm\in \mathbb{Z}^d \\ 0\leq j\leq d-1 \\ A\in [d]^{(j)}}} \FF\left(C_{\mm, A}^{d-j}, \lVert \cdot \rVert_\infty\right) \right )_{\ell_1}\oplus_1 \FF(\mathbb{Z}^d, \lVert \cdot \rVert_\infty).\]
\noindent Again by Proposition \ref{quotientquantitative}, each summand in the $\ell_1$ direct sum over $\mm, j$ and $A$ is $C_d$-isomorphic to a $C_d$-complemented subspace of $\FF(\mathbb{Z}^d, \lVert \cdot \rVert_\infty)$, so this implies that $\FF\left(\mathbb{Z}^d, \dk\right)$ is $C_d$-isomorphic to a $C_d$-complemented subspace of $\left (\bigoplus\limits_{n=1}^\infty \FF(\mathbb{Z}^d, \lVert \cdot \rVert_\infty) \right )_{\ell_1}$. 
By Proposition \ref{Z^d l_1 sum}, we obtain that $\FF\left(\mathbb{Z}^d, \dk\right)$ is $C_d$-isomorphic to a $C_d$-complemented subspace of $\FF(\mathbb{Z}^d, \lVert \cdot \rVert_\infty)$, as desired. \\

\noindent\textbf{2: Inductive proof for arbitrary $q$} \\

\noindent Assume that $q\geq 0$ and that the results of parts $(i)$ and $(ii)$ of Lemma \ref{lem2} hold for $q$. 
Fix $k\in \mathbb{N}$ and let \[M=\left(\mathbb{Z}^d\times \{(1+\varepsilon)^m:m\in \mathbb{N}\}^{q+1}, \dk\right).\] 
Let $A=\mathbb{Z}^d\times \{(1+\varepsilon)^m:m\in \mathbb{N}\}^q\times \{1+\varepsilon\}$, which is closed in $M$. 
Note that $A$ is a $1$-Lipschitz retract of $M$ via \[(p, (1+\varepsilon)^m)\mapsto (p, 1+\varepsilon) \qquad \text{for } p\in \left (\mathbb{Z}^d\times\{(1+\varepsilon)^n:n\in \mathbb{N}\}^q \right ) \text{ and }  m\in \mathbb{N}.\] \\
So by the remark following Lemma \ref{Quotient decomposition} with $L=1$, \[\refstepcounter{equation}\label{circle}\FF(M)\overset{4}{\sim}\FF\left (\bigslant{M}{A} \right )\oplus_1 \FF(A). \tag{\theequation}\]
Since $\FF(A)\cong \FF\left(\mathbb{Z}^d\times \{(1+\varepsilon)^m:m\in \mathbb{N}\}^q, \dk\right)$, by the induction hypothesis we have that $\FF(A)$ is $C_{\varepsilon, d, q}$-isomorphic to a $C_{\varepsilon, d, q}$-complemented subspace of $\FF(\mathbb{Z}^d, \lVert \cdot \rVert_\infty)$. \\

\noindent For $m\geq 2$, let $A_m=\mathbb{Z}^d\times \{(1+\varepsilon)^n:n\in \mathbb{N}\}^q\times \{(1+\varepsilon)^m\}\subset \bigslant{M}{A}$, and note that $\bigslant{M}{A}=\bigcup\limits_{m=2}^\infty (A_m\cup\{A\})$. 
The sequence $(A_m)_{m=2}^\infty$ is well-separated with respect to $A$ with constant depending only on $\varepsilon$: 
to verify this, given any $2\leq m<l$, any $(p, (1+\varepsilon)^m)\in A_m$ and any $(r, (1+\varepsilon)^l)\in A_l$, we have that
\begin{align*}
    &d_{\bigslant{M}{A}}\left ((p, (1+\varepsilon)^m), (r, (1+\varepsilon)^l) \right ) \\ =&\min\left\{d_M\left((p, (1+\varepsilon)^m), (r, (1+\varepsilon)^l)\right), d_M \left((p, (1+\varepsilon)^m), A \right)+d_M\left ((r, (1+\varepsilon)^l), A \right )\right\} \\=\refstepcounter{equation}\label{quotientmetric}&\min \left \{\min\left (\lVert (p, (1+\varepsilon)^m)-(r, (1+\varepsilon)^l) \rVert_\infty, k\right ), \min \left( (1+\varepsilon)^m-(1+\varepsilon), k\right)+ \min \left( (1+\varepsilon)^l-(1+\varepsilon), k\right) \right  \}. \tag{\theequation}
\end{align*}
\noindent The second term of \mbox{\eqcref{quotientmetric}} is at most $2\min \{(1+\varepsilon)^l, k\}$ and is equal to \[d_{\bigslant{M}{A}}\left ((p, (1+\varepsilon)^m), A \right )+d_{\bigslant{M}{A}}\left ((r, (1+\varepsilon)^l), A \right ).\]
The first term of \mbox{\eqcref{quotientmetric}} is
\begin{align*}
    &\geq \min\{(1+\varepsilon)^l-(1+\varepsilon)^m, k\} \\ &\geq\min\left\{(1+\varepsilon)^l-\frac{1}{1+\varepsilon}(1+\varepsilon)^l, k\right\} \\
    &\geq \frac{\varepsilon}{1+\varepsilon}\min\{(1+\varepsilon)^l, k\}. 
\end{align*}
So \[d_{\bigslant{M}{A}}\left ((p, (1+\varepsilon)^m), (r, (1+\varepsilon)^l) \right ) \geq \frac{\varepsilon}{2(1+\varepsilon)}\left (d_{\bigslant{M}{A}}\left ((p, (1+\varepsilon)^m), A \right )+d_{\bigslant{M}{A}}\left ((r, (1+\varepsilon)^l), A \right ) \right ).\]
\noindent Thus by combining Lemma \ref{well separated} and Lemma \ref{remove a point} $(ii)$, we obtain
\[\refstepcounter{equation}\label{circlecircle}\FF\left (\bigslant{M}{A} \right )\overset{C_\varepsilon}{\sim}\left (\bigoplus\limits_{m=2}^\infty \FF\left(A_m, d_{\bigslant{M}{A}}\right) \right )_{\ell_1}. \tag{\theequation} \]

\noindent Note that for $m\geq 2$, by the same computation as above, the metric on $A_m$ is 
\begin{align*}
    &d_{\bigslant{M}{A}}\left ( (p, (1+\varepsilon)^m), (r, (1+\varepsilon)^m) \right ) \\ =&\min \left (\min\{\lVert p-r \rVert_\infty, k\}, 2\min\{(1+\varepsilon)^m-(1+\varepsilon), k\} \right ) \\
    =&\min\{\lVert p-r\rVert_\infty, \min\{k, 2((1+\varepsilon)^m-(1+\varepsilon))\}\}.
\end{align*}
Let $r_m=\min\{k, 2((1+\varepsilon)^m-(1+\varepsilon))\}$ and observe that $r_m\geq \min\{1, 2((1+\varepsilon)^2-(1+\varepsilon))\}$. 
Thus for some $k_m\in \mathbb{N}$, this metric is $C_{\varepsilon}$-bilipschitz equivalent to $\left(\mathbb{Z}^d\times \{(1+\varepsilon)^n:n\in \mathbb{N}\}^q, d^{\lVert \cdot \rVert_\infty}_{k_m}\right)$ (take $k_m=\lceil r_m \rceil\leq C_\varepsilon r_m$). 
Combining this with the induction hypothesis gives that $\FF\left(A_m, d_{\bigslant{M}{A}}\right)$ is $C_{\varepsilon, d, q}$-isomorphic to a $C_{\varepsilon, d, q}$-complemented subspace of $\FF(\mathbb{Z}^d, \lVert \cdot \rVert_\infty)$. 
By \mbox{\eqcref{circle}}, \mbox{\eqcref{circlecircle}} and Proposition \ref{Z^d l_1 sum}, this implies that $\FF(M)$ is $C_{\varepsilon, d, q+1}$-isomorphic to a $C_{\varepsilon, d, q+1}$-complemented subspace of $\FF(\mathbb{Z}^d, \lVert \cdot \rVert_\infty)$. \\

\noindent Finally the argument for $M=(\mathbb{Z}^d\times \{(1+\varepsilon)^m:m\in \mathbb{N}\}^{q+1}, \lVert \cdot \rVert_\infty)$ follows exactly the same structure.
\noindent Define $A$ exactly as above, which is still a $1$-Lipschitz retract of $M$, so \mbox{\eqcref{circle}} still holds. Since
\[\FF(A)\cong \FF(\mathbb{Z}^d\times \{(1+\varepsilon)^m:m  \in \mathbb{N}\}^q, \lVert \cdot \rVert_\infty),\] by the induction hypothesis, $\FF(A)$ is $C_{\varepsilon, d, q}$-isomorphic to a $C_{\varepsilon, d, q}$-complemented subspace of $\FF(\mathbb{Z}^d, \lVert \cdot \rVert_\infty)$. 
Defining $(A_m)_{m=2}^\infty$ as subsets of $\bigslant{M}{A}$ as before, we still have that $(A_m)_{m=2}^\infty$ is well-separated with respect to $A$ with constant $\frac{\varepsilon}{2(1+\varepsilon)}$, so \mbox{\eqcref{circlecircle}} still holds. 
This time, for $m\geq 2$, the metric on $A_m$ is given by
\[d_{\bigslant{M}{A}}\left ((p, (1+\varepsilon)^m), (r, (1+\varepsilon)^m) \right )=\min\{\lVert p-r \rVert_\infty, 2((1+\varepsilon)^m-(1+\varepsilon))\},\]
\noindent which is still $C_{\varepsilon}$-bilipschitz equivalent to $\left(\mathbb{Z}^d\times \{(1+\varepsilon)^n:n\in \mathbb{N}\}^q, d^{\lVert \cdot \rVert_\infty}_{k_m}\right)$ for some $k_m\in \mathbb{N}$. 
We can therefore conclude this case in exactly the same way.
\end{proof}

\noindent Now we prove Theorem \ref{thm1}, which we restate for convenience.

\thmone*

\begin{proof} Throughout, unless otherwise specified, the metric on $\mathbb{R}^{d+q}$ and all of its subsets is taken to be the metric derived from the $\lVert \cdot \rVert_\infty$ norm. We will first prove the $d>0$ case, and the $d=0$ case will be a quick consequence. \\

\noindent\textbf{1: Proof in the case $d>0$} \\

\noindent We will show that $\FF(\mathbb{Z}^d)\comp \FF(N)$ and $\FF(N)\comp \FF(\mathbb{Z}^d)$, which gives the result by Pe\l{}czy\'{n}ski decomposition and Proposition \ref{Z^d l_1 sum}. \\

\noindent\textbf{1.1: Proof that $\FF(\mathbb{Z}^d)\comp \FF(N)$}\\

\noindent Let $M=\{(\dindex{a}{n_1}{(1)}, \cdots, \dindex{a}{n_d}{(d)}):n_i\in \mathbb{N}\}$. 
Then $\FF(M)\comp \FF(N)$ by Proposition \ref{Complementation}, so it suffices to show that $\FF(\mathbb{Z}^d)\comp \FF(M)$. \\

\noindent \noindent\textbf{1.1.1: Proof that $\FF(\mathbb{Z}^d)\comp \FF(M)$ in the case of equal sequences} \\

\noindent Assume for now that for each $1\leq i\leq d$, the sequence $(\dindex{a}{n}{(i)})_{n=1}^\infty$ is equal to $(a_n)_{n=1}^\infty$, where $0<a_1<a_2<\cdots$, $a_n\to \infty$ as $n\to \infty$ and $\frac{a_{n+1}}{a_n}\to 1$ as $n\to \infty$. \\

\noindent\textbf{Step 1: A construction} \\

\noindent We find arbitrarily late subsets of $\{a_n:n\in \mathbb{N}\}$ that are close to arithmetic progressions. 
\noindent Let $R>0$ and $k\in \mathbb{N}$ be arbitrary. We construct $n_1<\cdots<n_k$ in $\mathbb{N}$ with $a_{n_1}\geq R$ and $a_{n_k}\leq 1.01 a_{n_1}$, such that $\{a_{n_i}:i\in [k]\}$ is $4$-bilipschitz equivalent to an arithmetic progression. 
Let $\varepsilon\in (0, 1)$ be sufficiently small, to be determined later. 
Let $b_n=\frac{a_{n+1}}{a_n}>1$ for $n\in \mathbb{N}$, so that $b_n\to 1$ as $n\to \infty$. 
By discarding finitely many of the first terms of $(a_n)_{n=1}^\infty$, we may assume that $a_n\geq R$ for all $n$ and that $b_n\leq 1+\varepsilon$ for all $n$. \\

\noindent Since $b_n>1$ for all $n$ and $b_n\to 1$ as $n\to \infty$, we may choose $n_1\in \mathbb{N}$ such that $b_{n_1}=\underset{n\in \mathbb{N}}{\max}\ b_n$. 
If $1\leq i<k$ and $n_i$ has been chosen, then since $a_n\to \infty$ as $n\to \infty$, let $n_{i+1}\in \mathbb{N}$ be least such that $a_{n_{i+1}}\geq b_{n_1}a_{n_i}$. 
Then $n_{i+1}>n_i$ as $a_{n_{i+1}}>a_{n_i}$, and since $a_{n_{i+1}-1}<b_{n_1}a_{n_i}$, we have that $a_{n_{i+1}}=b_{n_{i+1}-1}a_{n_{i+1}-1}\leq  b_{n_1} a_{n_{i+1}-1}\leq b_{n_1}^2 a_{n_i}$. 
We obtain $n_1<\cdots<n_k$ such that $\forall 1\leq i\leq k-1$, $b_{n_1}a_{n_i}\leq a_{n_{i+1}}\leq b_{n_1}^2 a_{n_i}$. \\

\noindent We can write $b_{n_1}=1+\delta$ where $0<\delta\leq \varepsilon$. 
Choose $\varepsilon$ initially to be sufficiently small that this guarantees both $1+t\delta\leq (1+\delta)^t\leq 1+\frac{3}{2}t\delta$ for all $t\in \{0, 1, \cdots, 2k\}$, and $3(1+\delta)^{2(k-1)}\leq 3.03$. 
For $1\leq i<j\leq k$, we have that $(1+\delta)^{j-i}a_{n_i}\leq a_{n_j}\leq (1+\delta)^{2(j-i)}a_{n_i}$, so
\begin{align*}
    &(1+(j-i)\delta) a_{n_i}\leq a_{n_j}\leq (1+3(j-i)\delta)a_{n_i}, \text{ and hence } \\
    (j-i)\delta a_{n_1}\leq (j-&i)\delta a_{n_i}\leq a_{n_j}-a_{n_i}\leq 3(j-i)\delta a_{n_i}\leq 3(j-i)\delta (1+\delta)^{2(i-1)}a_{n_1}\leq 4(j-i)\delta a_{n_1}.
\end{align*}
\noindent Thus the map $(a_{n_i}\mapsto i\delta a_{n_1}$, $i\in [k])$ is a bilipschitz equivalence with distortion at most 4. Finally $a_{n_1}\geq R$ by our first reduction step, and since $3(1+\delta)^{2(k-1)}\leq 3.03$, we have that $a_{n_k}\leq (1+\delta)^{2(k-1)}a_{n_1}\leq 1.01a_{n_1}$. \\

\noindent\textbf{Step 2: Another short construction and proof in the reduced case} \\

\noindent We now select subsets as in step 1 which are very far apart and look at their $d$th powers in $M$ to show that $\left(\bigoplus\limits_{k=1}^\infty \FF\left(\{0, 1, \cdots, k\}^d\right) \right)_{\ell_1}\comp \FF(M)$. Together with Lemma \ref{Kalton decomposition}, this will complete the proof in the reduced case. \\

\noindent By inductively applying the construction of Step 1 above, we can find subsets $A_1, A_2, \cdots$ of $\{a_n:n\in \mathbb{N}\}$ such that the following hold.
\begin{itemize}
    \item $\forall k\in \mathbb{N}, \vert A_k \vert=k+1$.
    
    \item Each $A_k$ is bilipschitz equivalent to some arithmetic progression with distortion at most $4$.
    
    \item \refstepcounter{equation}\label{construction} If $a_k=\min A_k$, $b_k=\max A_k$, then $\forall k\in \mathbb{N}$, $a_{k+1}\geq 3b_k$ and $b_{k}\leq 1.01a_k$. \hfill (\theequation)
\end{itemize}
\noindent For $k\in \mathbb{N}$, let $M_k=A_k^d\subset M$. 
Since $a_{k+1}\geq 3b_k$ for all $k$, $(M_k)_{k=1}^\infty$ is a sequence of non-empty pairwise disjoint subsets of $\mathbb{R}^d$. 
By an easy computation, the third bullet point of \mbox{\eqcref{construction}} implies that $(M_k)_{k=1}^\infty$ is well-separated with respect to $0$ with constant $\frac{2/3}{1/3+1.01}$. 
 So by Lemma \ref{well separated disjoint}, \[\FF\left ( \left (\bigcup\limits_{k=1}^\infty M_k, \lVert \cdot \rVert_\infty \right ) \right )\sim \left (\bigoplus\limits_{k=1}^\infty \FF\left(M_k, \lVert \cdot \rVert_\infty\right) \right )_{\ell_1}\oplus_1 \ell_1.\] 
By Proposition \ref{Complementation}, $\FF\left (\left (\bigcup\limits_{k=1}^\infty M_k, \lVert \cdot \rVert_\infty \right ) \right )\comp \FF(M)$, and therefore  we get $\left (\bigoplus\limits_{k=1}^\infty \FF\left(M_k, \lVert \cdot \rVert_\infty\right) \right )_{\ell_1}\comp \FF(M)$. \\

\noindent For any $k$, $A_k$ is bilipschitz equivalent to some $k+1$-term arithmetic progression with distortion at most $4$, hence is bilipschitz equivalent to $\{0, 1, \cdots, k\}$ with distortion at most $4$.
Therefore $(M_k, \lVert \cdot \rVert_\infty)$ is bilipschitz equivalent to $(\{0, 1, \cdots, k\}^d, \lVert \cdot \rVert_\infty)$ with distortion at most $4$. 
Note here that we reduced to equal sequences in order to get a uniform bilipschitz equivalence, which might not have been guaranteed in the general case due to issues with different scalings in each coordinate. 
Thus $\FF\left(M_k, \lVert \cdot \rVert_\infty\right)\overset{4}{\sim} \FF\left(\{0, 1, \cdots, k\}^d, \lVert \cdot \rVert_\infty\right) \text{ for all }k$, so 
 \[\left (\bigoplus\limits_{k=1}^\infty \FF\left(M_k, \lVert \cdot \rVert_\infty\right) \right )_{\ell_1}\overset{4}{\sim}\left ( \bigoplus\limits_{k=1}^\infty \FF\left(\{0, 1, \cdots, k\}^d, \lVert \cdot \rVert_\infty\right) \right )_{\ell_1},\] giving
 \[\left ( \bigoplus\limits_{k=1}^\infty \FF\left(\{0, 1, \cdots, k\}^d, \lVert \cdot \rVert_\infty\right) \right )_{\ell_1}\comp \FF(M).\] \\

 \noindent By Lemma \ref{Kalton decomposition}, 
 \begin{align*}
     \FF(\mathbb{Z}^d)&\comp \left (\bigoplus\limits_{k=1}^\infty \FF \left (\overline{B}_{\lVert \cdot \rVert_\infty, \mathbb{Z}^d}(0, 2^k) \right ) \right )_{\ell_1} \\
     &=\left (\bigoplus\limits_{k=1}^\infty \FF \left(\{-2^k, -2^k+1, \cdots, 2^k\}^d, \lVert \cdot \rVert_\infty\right) \right )_{\ell_1} \\
     &\cong \left (\bigoplus\limits_{k=1}^\infty \FF \left(\{0, 1,  \cdots, 2^{k+1}\}^d, \lVert \cdot \rVert_\infty\right) \right )_{\ell_1}.
 \end{align*}
 \noindent So $\FF(\mathbb{Z}^d)\comp \left (\bigoplus\limits_{k=1}^\infty \FF \left(\{0, 1,  \cdots, 2^{k+1}\}^d, \lVert \cdot \rVert_\infty\right) \right )_{\ell_1}\comp \left ( \bigoplus\limits_{k=1}^\infty \FF\left(\{0, 1, \cdots, k\}^d, \lVert \cdot \rVert_\infty\right) \right )_{\ell_1}\comp \FF(M)$ as claimed. \ \ \ $\square$ \\

 \noindent\textbf{1.1.2: Proof that $\FF(\mathbb{Z}^d)\comp \FF(M)$ in the general case}  \\

 \noindent We argue that by shrinking $M$, we can reduce to the equal sequences case. 
More precisely, we show that there exists a sequence $(c_n)_{n=1}^\infty$ with $0<c_1<c_2<\cdots$, $c_n\to \infty$ as $n\to \infty$ and $\frac{c_{n+1}}{c_n}\to 1$ as $n\to \infty$, such that for every $i\in [d]$, $\{c_n:n\in \mathbb{N}\}$ is bilipschitz equivalent to some subsequence $A_i=\{\dindex{a}{\dindex[0pt]{j}{n}{(i)}}{(i)}:n\in \mathbb{N}\}$ of $\{\dindex{a}{n}{(i)}:n\in \mathbb{N}\}$.
 Assuming this, $M':=\{(c_{n_1}, \cdots,c_{n_d}):n_1, \cdots, n_d\in \mathbb{N}\}$ is bilipschitz equivalent to $M'':=A_1\times \cdots \times A_d\subset M$, so $\FF(M')\sim \FF(M'')\comp \FF(M)$ by Proposition \ref{Complementation}. 
 By the equal sequences case, $\FF(\mathbb{Z}^d)\comp \FF(M')$, so we obtain $\FF(\mathbb{Z}^d)\comp \FF(M)$. We now show the existence of $(c_n)_{n=1}^\infty$.\\

 \noindent All we need to do is construct a sequence $(c_n)_{n=1}^\infty $ that grows sufficiently quickly compared to all $(\dindex{a}{n}{(i)})_{n=1}^\infty$, and then select nearby points in each $(\dindex{a}{n}{(i)})_{n=1}^\infty$. 
 For $i\in [d]$, $n\in \mathbb{N}$, we write $\frac{\dindex{a}{n+1}{(i)}}{\dindex{a}{n}{(i)}}=1+\dindex{\varepsilon}{n}{(i)}$, where $\dindex{\varepsilon}{n}{(i)}>0$ and for each $i\in [d]$, $\dindex{\varepsilon}{n}{(i)}\to 0$ as $n\to \infty$. 
 By discarding finitely many of the first terms of each $(\dindex{a}{n}{(i)})_{n=1}^\infty$, we can assume that $\forall n\in \mathbb{N}$, $\forall i\in [d]$, $\dindex{\varepsilon}{n}{(i)}<\frac{1}{10^{100}}$.
 By scaling, we can assume $\dindex{a}{1}{(i)}=1$ for all $i$. 
 Let $r_n=1+1000\left (\sum\limits_{i=1}^d\underset{\substack {m \geq n-10 \\ m\in \mathbb{N}}}{\max}\dindex{\varepsilon}{m}{(i)}\right )$ for each $n\in \mathbb{N}$, and note that $r_n\to 1$ as $n\to \infty$. \\
 
 \noindent Let $c_1=1$, and for $n\geq 1$, let $c_{n+1}=r_n c_n$. 
 Then $\frac{c_{n+1}}{c_n}\to 1$ as $n\to \infty$, and $0<c_1<c_2<\cdots$. 
 Since $\forall n$, $\frac{c_{n+1}}{c_n}=r_n\geq 1+\dindex{\varepsilon}{n}{(1)}$, we deduce that $c_n\geq \dindex{a}{n}{(1)}$ for all $n$, so $c_n\to \infty$ as $n\to \infty$. 
Fix $i\in [d]$. We extract the desired subsequence $\left(\dindex{a}{\dindex{j}{n}{(i)}}{(i)}\right)_{n=1}^\infty$ of $(\dindex{a}{n}{(i)})_{n=1}^\infty$. To avoid cluttering the notation, since $i$ is fixed, we will write $j_n$ instead of $\dindex{j}{n}{(i)}$ in the following with the understanding that $j_n$ depends on $i$. \\
 
 \noindent Let $j_1=1$, and for $p\geq 2$, let $j_p\in \mathbb{N}$ be least with $\dindex{a}{j_p}{(i)}\geq c_p$ (as $\dindex{a}{n}{(i)}\to \infty$ as $n\to \infty$). 
 Note that as $c_2>c_1=\dindex{a}{1}{(i)}$, we have $j_2>1=j_1$. 
 For $p\geq 3$, assuming $j_1<\cdots<j_{p-1}$, we have $j_{p-1}\geq p-1$ and $\dindex{a}{j_{p-1}-1}{(i)}<c_{p-1}$, so 
 \begin{align*}
     \dindex{a}{j_{p-1}}{(i)}<c_{p-1}(1+\dindex{\varepsilon}{j_{p-1}-1}{(i)})\leq c_{p-1}(1+\underset{\substack{m\geq p-11 \\ m\in \mathbb{N}}}{\max}\dindex{\varepsilon}{m}{(i)})\leq c_{p-1}\cdot r_{p-1}=c_p,
 \end{align*}
 and therefore $j_p>j_{p-1}$. 
This constructs $j_1<j_2<\cdots$. \\

 \noindent For $t\in \mathbb{N}$, $\dindex{a}{j_t}{(i)}\geq c_t$. If $t\geq 2$, then $j_t\geq t\geq 2$, so by minimality of $j_t$ and since $j_t-1\geq t-1$, we have 
 \begin{align*} \dindex{a}{j_t}{(i)}=(1+\dindex{\varepsilon}{j_t-1}{(i)})\dindex{a}{j_t-1}{(i)}< (1+\dindex{\varepsilon}{j_t-1}{(i)})c_t\leq (1+\underset{m\geq t-1}{\max}\dindex{\varepsilon}{m}{(i)})c_t, \end{align*} and we note $\dindex{a}{j_t}{(i)}\leq (1+\underset{m\geq t-1}{\max}\dindex{\varepsilon}{m}{(i)})c_t$ also holds for $t=1$.
 Now fix $p, q\in \mathbb{N}$ with $p<q$ and compare $\dindex{a}{j_q}{(i)}-\dindex{a}{j_p}{(i)}$ to $c_q-c_p$. For an upper bound, 
 \[\dindex{a}{j_q}{(i)}-\dindex{a}{j_p}{(i)}\leq (1+\underset{m\geq q-1}{\max}\dindex{\varepsilon}{m}{(i)})c_q-c_p\leq 2(c_q-c_p),\]
 where the last inequality, being equivalent to $(1-\underset{m\geq q-1}{\max} \dindex{\varepsilon}{m}{(i)})c_q\geq c_p$, holds since
 \begin{align*}
     (1-\underset{m\geq q-1}{\max} \dindex{\varepsilon}{m}{(i)})c_q & \geq (1-\underset{m\geq q-1}{\max} \dindex{\varepsilon}{m}{(i)})(1+1000\underset{m\geq q-11}{\max}\dindex{\varepsilon}{m}{(i)})c_{q-1} \\
     &\geq (1-\underset{m\geq q-1}{\max} \dindex{\varepsilon}{m}{(i)})(1+1000\underset{m\geq q-1}{\max}\dindex{\varepsilon}{m}{(i)})c_p \\
     &\geq c_p,
 \end{align*}
 since if $t:=\underset{m\geq q-1}{\max}\dindex{\varepsilon}{m}{(i)}$, we have $t\in (0, \frac{1}{10^{100}}]$ and so $(1-t)(1+1000t)\geq 1$. \\

 \noindent Also, \[\dindex{a}{j_q}{(i)}-\dindex{a}{j_p}{(i)}\geq c_q-(1+\underset{m\geq p-1}{\max}\dindex{\varepsilon}{m}{(i)})c_p\geq \frac{1}{2}(c_q-c_p),\] where the last inequality, which is equivalent to ${c_q\geq (1+2\underset{m\geq p-1}{\max}\dindex{\varepsilon}{m}{(i)})c_p}$, holds as
 \[c_q\geq c_{p+1}\geq (1+1000\underset{m\geq p-10}{\max}\dindex{\varepsilon}{m}{(i)})c_p\geq (1+2\underset{m\geq p-1} {\max}\dindex{\varepsilon}{m}{(i)})c_p.\]
 \noindent So the map $(c_p\mapsto \dindex{a}{j_p}{(i)}$, $p\in \mathbb{N})$ is a bilipschitz equivalence with distortion at most $4$. $\square$ \\
 
\noindent This completes the proof that $\FF(\mathbb{Z}^d)\comp \FF(N)$.  \\

 \noindent\textbf{1.2: Proof that $\FF(N)\comp \FF(\mathbb{Z}^d)$} \\

 \noindent We prove this using Lemmas \ref{lem2} and \ref{Kalton decomposition}.
 \noindent We first enlarge $N$ in order to apply Lemma \ref{lem2}. 
 \noindent To do this, we find $\varepsilon>0$ such that for each $ 1\leq j\leq q$, $\{\dindex{b}{m}{(j)}:m\in \mathbb{N}\}$ is bilipschitz equivalent to a subset of $\{(1+\varepsilon)^m:m\in \mathbb{N}\}$. All we need to do is to choose $\varepsilon$ very small, so that $\{(1+\varepsilon)^m:m\in \mathbb{N}\}$ is much finer than each $(\dindex{b}{m}{(j)})_{m=1}^\infty$, and then select close points in each $(\dindex{b}{m}{(j)})_{m=1}^\infty$. 
Since $\underset{m\to \infty}\liminf \frac{\dindex{b}{m+1}{(j)}}{\dindex{b}{m}{(j)}}>1$ for all $j$, we can choose $\varepsilon>0$ small enough such that $\forall m\in \mathbb{N}$, $\forall 1\leq j\leq q$, $\frac{\dindex{b}{m+1}{(j)}}{\dindex{b}{m}{(j)}}\geq (1+\varepsilon)^{100}$, and such that $(1+\varepsilon)^{98}(1-\varepsilon)\geq 1$. 
By scaling, we can assume that for all $m$ and  $j$, $\dindex{b}{m}{(j)}\geq (1+\varepsilon)^{100}$. \\
 
 \noindent Fix $1\leq j\leq q$. For each $m\in \mathbb{N}$, let $k_m\in \mathbb{N}$ be least such that $\dindex{b}{m}{(j)}\leq (1+\varepsilon)^{k_m}$ (here $k_m$ depends on $j$, this is suppressed from the notation). 
 Then $k_m\geq 2$ since $\dindex{b}{m}{(j)}\geq (1+\varepsilon)^{100}$, so $\dindex{b}{m}{(j)}\geq (1+\varepsilon)^{k_m-1}$. 
 For every $m$, $(1+\varepsilon)^{k_{m+1}-k_m+1}\geq \frac{\dindex{b}{m+1}{(j)}}{\dindex{b}{m}{(j)}}\geq (1+\varepsilon)^{100}$, so $k_{m+1}-k_m\geq 99$. 
 For any $n, m \in \mathbb{N}$ with $n<m$, we have \[\dindex{b}{m}{(j)}-\dindex{b}{n}{(j)}\leq (1+\varepsilon)^{k_m}-(1+\varepsilon)^{k_n-1}\leq 2((1+\varepsilon)^{k_m}-(1+\varepsilon)^{k_n}),\]
\noindent where the last inequality, which is equivalent to $(1+\varepsilon)^{k_m-k_n}\geq 2-(1+\varepsilon)^{-1},$ holds since
\[2-(1+\varepsilon)^{-1}=1+\frac{\varepsilon}{1+\varepsilon}\leq 1+\varepsilon\leq (1+\varepsilon)^{99}\leq (1+\varepsilon)^{k_m-k_n}.\]

\noindent Also, we have
\[\dindex{b}{m}{(j)}-\dindex{b}{n}{(j)}\geq (1+\varepsilon)^{k_m-1}-(1+\varepsilon)^{k_n}\geq \frac{1}{2} ((1+\varepsilon)^{k_m}-(1+\varepsilon)^{k_n}),\] \noindent where the last inequality, being equivalent to $2(1+\varepsilon)^{k_m-k_n-1}-(1+\varepsilon)^{k_m-k_n}\geq 1,$ holds since
\[2(1+\varepsilon)^{k_m-k_n-1}-(1+\varepsilon)^{k_m-k_n}=(1+\varepsilon)^{k_m-k_n}\left (\frac{1-\varepsilon}{1+\varepsilon} \right )\geq (1+\varepsilon)^{98}(1-\varepsilon)\geq 1.\]

\noindent Thus the map $(\dindex{b}{m}{(j)}\mapsto (1+\varepsilon)^{k_m}, m\in \mathbb{N})$ is a bilipschitz equivalence with distortion at most $4$. \\

\noindent So fix $\varepsilon>0$ such that $N$ is bilipschitz equivalent to a subset of \[M:=\{(\dindex{a}{n_1}{(1)}, \cdots, \dindex{a}{n_d}{(d)}):n_i\in \mathbb{N}\}\times \{(1+\varepsilon)^m:m\in \mathbb{N}\}^q.\] \\
By Proposition \ref{Complementation}, $\FF(N)\comp \FF(M)$, so it suffices to show that $\FF(M)\comp \FF(\mathbb{Z}^d)$, which now follows from Lemmas \ref{lem2} and \ref{Kalton decomposition}. 
Let $p\in M$ be a basepoint. By Lemma \ref{Kalton decomposition}, we have that $\FF(M)\comp \left ( \bigoplus\limits_{k=1}^\infty \FF\left(\overline{B}_M(p, 2^k)\right) \right )_{\ell_1}$. 
For each $k\in \mathbb{N}$, $\overline{B}_M(p, 2^k)\subset A_k\times \{(1+\varepsilon)^m:m\in \mathbb{N}\}^q$ for some finite subset $A_k$ of $\mathbb{R}^d$. 
Thus there is a finite $B_k\subset \mathbb{Z}^d$ such that $\overline{B}_M(p, 2^k)$ is $2$-bilipschitz equivalent to a subset of $B_k\times \{(1+\varepsilon)^m:m\in \mathbb{N}\}^q\subset \mathbb{Z}^d\times \{(1+\varepsilon)^m:m\in \mathbb{N}\}^q$. \\

\noindent By Proposition \ref{quotientquantitative} and Lemma \ref{lem2}, this implies that for some constant $C=C_{\varepsilon, d, q}\geq 1$, we have for each $k$ that $\FF\left(\overline{B}_M(p, 2^k)\right)$ is $C$-isomorphic to a $C$-complemented subspace of $\FF\left(\mathbb{Z}^d, \lVert \cdot \rVert_\infty \right)$, so $\FF(M)\comp \left (\bigoplus\limits_{k=1}^\infty \FF \left (\mathbb{Z}^d, \lVert \cdot \rVert_\infty \right ) \right )_{\ell_1}$.
By Proposition \ref{Z^d l_1 sum}, we deduce that $\FF(M)\comp \FF(\mathbb{Z}^d)$. $\square$ \\
 
\noindent Theorem \ref{thm1} in the case $d>0$ now follows from Proposition \ref{Z^d l_1 sum} and Pe\l{}czy\'{n}ski decomposition. It remains to deal with the $d=0$, $q>0$ case. \\

\noindent\textbf{2: Proof in the case $d=0$, $q>0$} \\

\noindent By Proposition \ref{Complementation} and the $d=1$ case, \[\FF(N)\comp\FF(\mathbb{Z}\times \{(\dindex{b}{m_1}{(1)}, \cdots, \dindex{b}{m_q}{(q)}):m_j\in \mathbb{N}\})\sim \FF(\mathbb{Z}),\] and so $\FF(N)\comp \ell_1$ (since $\FF(\mathbb{Z})\cong \ell_1$). On the other hand, by Proposition \ref{l1}, $\ell_1\comp \FF(N)$. So by Pe\l{}czy\'{n}ski decomposition, $\FF(N)\sim \ell_1$. 

\end{proof}

\noindent This gives the following corollary.

\begin{corollary}
    Let $p$ be a nonconstant real polynomial. Then
    \[\FF\left(\{\left(p(n), p(m)\right):n, m\in \mathbb{N}\} \right )\sim \FF(\mathbb{Z}^2).\]
\end{corollary}

\noindent We remark that the same methods can also be used to prove the following.

\begin{proposition}\label{extra}
    Let $d\in \mathbb{N}$, $q\in \mathbb{Z}_{\geq 0}$.  For each $1\leq j\leq q$, let $(\dindex{b}{m}{(j)})_{m=1}^\infty$ be a sequence such that $0<\dindex{b}{1}{(j)}<\dindex{b}{2}{(j)}<\cdots$ and $\underset{m\to \infty}{\liminf}{\frac{\dindex{b}{m+1}{(j)}}{\dindex{b}{m}{(j)}}}>1$. Then
    \[\FF \left (\mathbb{R}^d\times\{(\dindex{b}{m_1}{(1)}, \cdots, \dindex{b}{m_q}{(q)}):m_j\in \mathbb{N}\} \right )\sim \FF(\mathbb{R}^d).\]
\end{proposition}

\noindent Only very minor changes are required. If in the statement of Lemma \ref{lem2}, each occurence of $\left(\mathbb{Z}^d, \lVert \cdot \Vert_\infty\right)$ is replaced by $\left(\rr^d, \lVert \cdot \rVert_\infty\right)$, then the same conclusion still holds, with three minor modifications to the proof. In this case, we instead define the $C_{\mm, A}^{d-j}$ to be continuous cubes rather than discrete cubes. In the proof of $w^\star$-to-$w^\star$ continuity, pointwise-to-pointwise continuity on $B_{\lip\left(\mathbb{R}^d, \dk\right)}$ of each $\lip\left(\bigslant{C_{\mm, A}^{d-j}}{\partial C_{\mm, A}^{d-j}}\right)$ coordinate follows from the fact that each $S_{\mm, A}^{d-j}$ is pointwise-to-pointwise continuous on norm bounded sets. Finally, instead of Proposition \ref{Z^d l_1 sum}, we instead use that $\FF\left(\rr^d, \lVert \cdot \rVert_\infty\right)\sim \left(\bigoplus\limits_{k=1}^\infty \FF\left(\rr^d, \lVert \cdot \rVert_\infty\right) \right)_{\ell_1}$. This holds more generally with $(\mathbb{R}^d, \lVert \cdot \rVert_\infty)$ replaced by any Banach space, which is a result of Kaufmann (see \cite{Kaufmann}*{Theorem 3.1}). The same argument as in the proof of Theorem \ref{thm1} then shows that 
\[\FF \left (\mathbb{R}^d\times\{(\dindex{b}{m_1}{(1)}, \cdots, \dindex{b}{m_q}{(q)}):m_j\in \mathbb{N}\} \right )\comp \FF(\rr^d),\] and the reverse complementation holds by Proposition \ref{Complementation}, so the result follows by Pe\l{}czy\'{n}ski decomposition.
\section{Open problems}

\noindent We collect some interesting open problems from the literature on Lipschitz-free spaces over subsets of $\rr^d$, and one problem related to a result used in this paper. We refer the reader to the papers cited in this section for discussions of these questions and for further open problems. Some fundamental questions about Lipschitz-free spaces over $\rr^d$ and over its subsets remain open. It is known that $\FF(\rr)$ and $\FF(\rr^2)$ are not isomorphic (see \cite{planarearthmover}), but the corresponding non-isomorphism result for two distinct dimensions, both at least $2$, is unknown.
\begin{problem}
    Are $\FF(\rr^2)$ and $\FF(\rr^3)$ isomorphic? If $m, n\geq 2$ are distinct, are $\FF(\rr^n)$ and $\FF(\rr^m)$ ever isomorphic?
\end{problem}
\noindent Another problem, raised in \cite{Aliaga}*{Question 5}, relates to isomorphism classes of Lipschitz-free spaces over subsets of $\rr^2$. It is known that $\FF(\rr\times \mathbb{Z})$ is not isomorphic to any of $\FF(\rr)$, $\FF(\mathbb{Z})$ or $\FF(\mathbb{Z}^2)$. In the first case, the spaces are not isomorphic because their duals are not isomorphic. In the second and third cases, $\FF(\rr\times \mathbb{Z})$ fails the Radon-Nikod\'ym property, while both $\FF(\mathbb{Z})$ and $\FF(\mathbb{Z}^2)$ have the Radon Nikod\'ym property. It is, however, unknown whether $\FF(\rr\times \mathbb{Z})$ is isomorphic to $\FF(\rr^2)$. By comparison, Proposition \ref{extra} shows that whenever a sparse sequence factor $A$ is added, then $\FF(\rr\times A)\sim \FF(\rr)$.
\begin{problem}[\cite{Aliaga}*{Question 5}]
    Are $\FF(\rr\times \mathbb{Z})$ and $\FF(\rr^2)$ isomorphic?
\end{problem}

\noindent For the corresponding problem for Lipschitz spaces, Aliaga poses the following problem in \cite{Aliaga}*{Question 3}.

\begin{problem}
    Let $M\subset \mathbb{R}^2$ be infinite. Does it follow that either $\lip(M)\sim \lip(\rr)$ or $\lip(M)\sim \lip(\rr^2)$?
\end{problem}

\noindent Throughout this paper, we repeatedly used the fact that $\FF(\mathbb{Z}^d)\sim \left( \bigoplus\limits_{k=1}^\infty \FF(\mathbb{Z}^d) \right)_{\ell_1}$. The analogous isomorphism is known to hold for several more general classes of metric spaces, including nets in Banach spaces  (see \cite{normedspaces}*{Theorem 3.6}) and certain homogeneous, unbounded metric spaces (\cite{normedspaces}*{Theorem 3.2}). The latter result generalises the result of Kaufmann for Banach spaces (\cite{Kaufmann}*{Theorem 3.1}). In \cite{normedspaces}*{Question 2}, it was asked whether there is any infinite metric space failing this property.

\begin{problem}[\cite{normedspaces}*{Question 2}]
    Given an infinite metric space $M$, are $\FF(M)$ and $\left(\bigoplus\limits_{k=1}^\infty \FF(M) \right)_{\ell_1}$ isomorphic?
\end{problem}

\section*{Acknowledgements}
\noindent The author is supported by a PhD studentship from the Department of Pure Mathematics and Mathematical Statistics at the University of Cambridge, funded by XTX Markets. This work was carried out as part of a PhD under the supervision of Andr\'as Zs\'ak, who the author would like to thank for his guidance and support.

\end{document}